\documentclass[12pt]{article}

\usepackage{amsmath}
\usepackage{amsthm}
\usepackage{amssymb}
\usepackage{cite}
\usepackage{url}
\usepackage[multiple]{footmisc}
\usepackage{graphicx,subcaption}
\usepackage{epstopdf}
\usepackage{chngcntr}
\usepackage{enumitem}
\usepackage{hhline}
\usepackage{array}
\usepackage{hyperref}
\usepackage[capitalise, noabbrev]{cleveref}
\usepackage{color}
\usepackage{array}
\usepackage{multirow}

\if@twoside
\else 
\fi

\numberwithin{equation}{section}
\counterwithin{figure}{section}
\counterwithin{table}{section}

\theoremstyle{plain}
\newtheorem{theorem}{Theorem}[section]
\newtheorem{lemma}[theorem]{Lemma}
\newtheorem{proposition}[theorem]{Proposition}

\newtheorem{corollary}[theorem]{Corollary}
\newtheorem{problem}{Problem}
\newtheorem*{problem*}{Problem}
\newtheorem{method}{Method}
\theoremstyle{definition}
\newtheorem{definition}{Definition}[section]

\theoremstyle{remark}
\newtheorem{remark}{Remark}[section]

\crefname{hypothesis}{Hypothesis}{Hypotheses}
\crefname{problem}{Problem}{Problems}
\crefname{method}{Method}{Methods}
\crefname{section}{Section}{Sections}

\newcommand{\abs}[1]{\left\vert#1\right\vert}
\newcommand{\spr}[1]{\left\langle\,#1\,\right\rangle}
\newcommand{\kl}[1]{\left(#1\right)}

\definecolor{aog}{rgb}{0.0, 0.5, 0.0}

\newcommand{\R}{\mathbb{R}} 
\newcommand{\C}{\mathbb{C}}

\newcommand{\mfunc}[1]{\text{#1} \,}

\newcommand{\sgn}{\mfunc{sgn}}

\newcommand{\eps}{\varepsilon}

\usepackage[table,xcdraw]{xcolor}

\title{An Inverse Problems Approach to Pulse Wave Analysis in the Human Brain}
\author{
Lukas Weissinger\footnote{ Corresponding author.} 
\footnote{Johann Radon Institute Linz, Altenbergerstra{\ss}e 69, A-4040 Linz, Austria, (lukas.weissinger@ricam.oeaw.ac.at, simon.hubmer@ricam.oeaw.ac.at, ronny.ramlau@ricam.oeaw.ac.at)},
Simon Hubmer\footnotemark[2]~\footnote{Johannes Kepler University Linz, Institute of Industrial Mathematics, Altenbergerstra{\ss}e 69, A-4040 Linz, Austria, (simon.hubmer@jku.at, ronny.ramlau@jku.at)},
\\
Ronny Ramlau \footnotemark[2]~\footnotemark[3] ,
Henning U.\ Voss\footnote{Cornell University - College of Human Ecology, Cornell MRI Facility, Martha Van Rensselaer Hall, Ithaca NY 14870, United States of America, (hv28@cornell.edu)}
}

\begin{document}

\maketitle

\begin{abstract}
Cardiac pulsations in the human brain have received recent interest due to their possible role in the pathogenesis of neurodegenerative diseases. Further interest stems from their possible application as an endogenous signal source that can be utilized for brain imaging in general. The (pulse-)wave describing the blood flow velocity along an intracranial artery consists of a forward (anterograde) and a backward (retrograde, reflected) part, but measurements of this wave usually consist of a superposition of these components. In this paper, we provide a mathematical framework for the inverse problem of estimating the pulse wave velocity, as well as the forward and backward component of the pulse wave separately from MRI measurements on intracranial arteries. After a mathematical analysis of this problem, we consider possible reconstruction approaches, and derive an alternate direction approach for its solution. The resulting methods provide estimates for anterograde/retrograde wave forms and the pulse wave velocity under specified assumptions on a cerebrovascular model system. Numerical experiments on simulated and experimental data demonstrate the applicability and preliminary in vivo feasibility of the proposed methods.
\end{abstract}


\smallskip
\noindent \textbf{MSC-Codes.} 47J06, 65J22, 65J20, 47A52

\section{Introduction}\label{sec:intro}

The vascular network is a complex biophysical system optimized for an efficient perfusion of the brain with blood. Its arterial tree consists of blood vessels whose diameters range over three orders of magnitude and, in the healthy brain, are perfectly tuned to absorb the cardiac pressure pulsations to minimize blood flow resistance and backflow, and thus maximize perfusion \cite{RN3496, RN3663, RN3480}. In the brain, cerebrovascular pulsatility \cite{Zou_Park_Kelly_Egnor_Wagshul_Madsen_2008} has become an important research topic due to its possible role in ageing-related brain health. Notably, after the age of 60, increases in blood pressure are mainly attributable to a rapid increase of pulse pressure (the difference between systolic and diastolic pressure), driven by the increase in systolic pressure \cite{RN6068}. As we get older, and in some conditions such as chronic hypertension, the physical properties of the vascular network, including its viscoelasticity, change \cite{RN3540, RN4338, RN5790, RN4339, RN3507}. Ageing-related changes of the central and peripheral cardiovascular system affect blood-flow, pulsatility, and pathologies of the brain, too \cite{RN3476, RN6080, RN6070}. The resulting disturbance of its fine-tuning reduces the absorption of blood pressure pulsations, increasing resistance, and possibly causing damage including hemorrhagic stroke and microbleeds. In particular, small-scale changes not visible in conventional in-vivo imaging \cite{RN4508, RN3545, RN741, RN742, RN4511} have garnered attention as possible causes of dementia \cite{RN5339, RN3476, RN5338, RN5337, RN5341}. In addition, the pressure gradient caused by arterial pulsatility is thought to be a driving force in the brain’s paravascular waste clearance system \cite{RN5618, RN5869, RN5815, RN5619}, which affects brain health by transporting away metabolic waste products. 

 Arterial blood flow consists of a steady and a pulsatile flow component \cite{RN5453, RN3496, RN3480}. The former is affected by vascular resistance, the latter by arterial stiffness. As the pulse wave travels through the brain, it changes both in shape and amplitude \cite{RN3496}. It is well-known that this pulse wave is reflected at vessel branches and endings and therefore both shape and amplitude of the wave measured in a single point on an artery need to be considered as resulting from a superposition of anterograde and retrograde waves. The anterograde wave is a product of the pumping action of the heart and distortions experienced up to the point of measurement. The retrograde wave is a product of the distortions experienced by the downstream vasculature including wave reflection. 
 The induced (pulse) wave travels within the blood vessels at a velocity called the \emph{pulse wave velocity} (PWV). In the main cerebral arteries, the PWV is typically higher in magnitude (up to about 12 m/s \cite{RN5528}) than the blood flow velocity, (about 30 to 100 cm/s) which varies widely with age, cardiac phase, blood pressure and type of artery, among other factors \cite{RN5460,RN6220}. 
Due to its high magnitude, the PWV is currently not directly measurable with MRI techniques. However,
4D Flow MRI allows to measure blood flow velocity over a specific field of view over several time steps in the cardiac cycle, resulting in data equivalent to a time-dependent pulse wave in each three-dimensional voxel of the acquired images, see \cref{fig:principle}.
    \begin{figure}[ht!]
        \centering
        \includegraphics[height=.23\textheight]{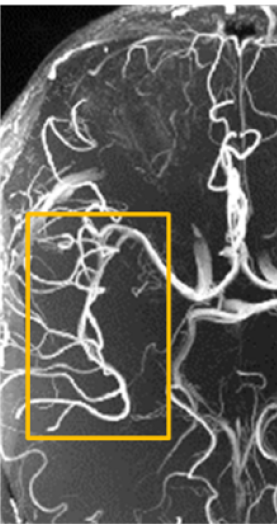}
        \includegraphics[height=.23\textheight]{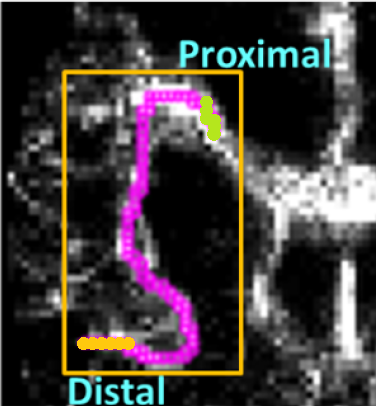}
        \includegraphics[height=.23\textheight]{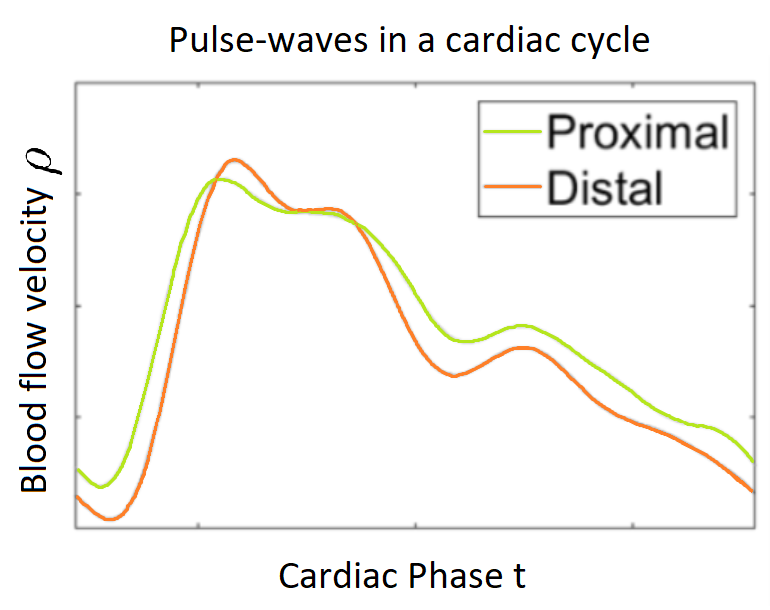}
        \caption{High resolution angiogram (left), tracked artery in 4D Flow MRI images (middle), and averaged pulse waves over color-coded voxels (right).}
        \label{fig:principle}
    \end{figure}

Recently, there has been lots of progress in non-invasive imaging of cerebral pulsatility \cite{RN5360, RN6069, RN3466, Hubmer_Neubauer_Ramlau_Voss_2018, Hubmer_Neubauer_Ramlau_Voss_2020, RN5342, RN3467, RN3468, RN3602, RN5298, RN5334} and pulse waves \cite{RN5802, RN3481, Voss_2018, RN5795, RN5614} with MRI. In particular, characterizing the cerebrovascular network by means of its pulsatile properties has come within reach. The fundamental parameter used to describe vascular networks is the impedance. It has specific analogs in the electrical equivalent circuit model \cite{Hademenos_Massoud_1997,Korürek_Yildiz_Yüksel_2010,Lim_Chan_Dokos_Ng_Latif_Vandenberghe_Karunanithi_Lovell_2013,Murray_1964,Rupnik_Runovc_Sket_Kordas_2002}, which are summarized in \cref{table:circuit}.
\begin{table}[ht!]
\caption{The electrical equivalent circuit model}
    \begin{tabular}{lllll}
    Hemodynamic parameter & Electrical equivalence &  &  &  \\\hline
    V, volume                                     & Q, charge                                      &  &  \\
    p, pressure                                   & U, voltage                                     &  &  \\
    R, vascular   resistance                      & R, resistance                                  &  &  \\
    L, inertance                                  & L, inductance                                  &  &  \\
    C, compliance; C   = dV/dp                    & C, capacitance; C   = dQ/dU                    &  &  \\
    f, volumetric bl.   flow rate; f = dV/dt      & I, current; I =   dQ/dt                        &  &  \\
    Poiseuille’s law;   f = $\Delta$p/R          & Ohm's law; I = U/R                     &  & \\
    Z, impedance; Z$(\omega) = p(\omega)/f(\omega)\in \C$ & Z, impedance; Z$(\omega) = U(\omega)/I(\omega)\in \C$ & &
    \end{tabular}
    
    \label{table:circuit}
\end{table}
Measuring the impedance of the vascular tree would determine vascular resistance as well as compliance, linking the PWV with blood flow \cite{Rabben_Stergiopulos_Hellevik_Smisetz_slordahl_2004, Villiemoz_Stergiopulos_Meuli_2002}. This equivalence allows one to apply system identification procedures used in electrical engineering to the vascular network, placing a vast array of computational and analytical methods at our disposal to identify changes in arterial stiffness. 

From \cref{table:circuit} one can see that in order to measure impedance directly, one would have to measure blood pressure and flow at the same location and time. However, noninvasive direct measurement of impedance in the brain is presently not possible. As a workaround, rather than measuring pressure and flow in one location, a signal such as phase contrast or blood volume could be measured at multiple points along the artery \cite{Hubmer_Neubauer_Ramlau_Voss_2018, RN5342, RN3467, RN3468, RN3602,Voss_2018,RN5614, RN5298, RN5334,RN3663,Zou_Park_Kelly_Egnor_Wagshul_Madsen_2008}. This way, the PWV has recently been estimated under the approximation that reflected waves can be neglected \cite{bjornfot2021,Rivera2020}.

In this paper, we consider the problem of splitting pulse waves in the human brain along intracranial arteries into their forward and backward components. We consider the situation that the pulse wave is observable with MRI, for example as a volume- or flow-wave \cite{RN5870}. This restriction precludes prior approaches that are based on measuring the impedance of a blood vessel \cite{RN7124, RN7125, RN5794, RN3509, RN6926, RN5632, RN5481, RN7138}; the calculation of impedance (Table~\ref{table:circuit}) requires the observation of blood pressure at the same time and location as blood flow, for which no MRI pulse sequence has been developed yet. The price to pay for replacing blood pressure is that we will have to measure the pulse wave at least at two separate locations in the brain, which necessitates the estimation of pulse wave velocity between these locations, too \cite{RN5782}. The PWV has been estimated recently under the approximation that reflected waves can be neglected \cite{bjornfot2021,Rivera2020}.
Thus, our approach includes an alternative way for PWV estimation inside the brain. 
The transition from two system variables to one variable at two locations is akin to the time-delayed embedding of a system, or the conversion of a two-dimensional system of differential equations to a one-dimensional differential equation of higher order. A first approach to pulse wave splitting using this rationale has been considered in the master's thesis \cite{Augl_Thesis_2022}, and other one-variable approaches are provided by Refs. \cite{RN7134, RN6919}. 

In our paper, we consider the MRI pulse wave data as a function $p(x,t):\Omega \times \R \to \R$ defined on a domain $\Omega \subset \R^3$ representing the given field of view. We assume that $p(x,t)$ is a periodic function in time $t$ with period $T$, corresponding to one cardiac cycle $[0,T] \subset \R$.
For spatial points $x_1,x_2\in\Omega$ in a blood vessel segment of our interest we will subsequently use the notation $p_1(t):=p(x_1,t)$ and $p_2(t):=p(x_2,t),$ respectively.
If we consider $p_1$ and $p_2$, proximal and distal to the heart, along a non-branching vessel segment of constant diameter, each of these waves is a superposition of forward ($p_{1f},p_{2f}$) and backward ($p_{1b},p_{2b}$) components, i.e.,    
\begin{equation}\label{eq_two_waves}
        p_1(t) = p_{1f}(t) + p_{1b}(t) \,,
        \qquad \text{and}\qquad
        p_2(t) = p_{2f}(t) + p_{2b}(t) \,.
    \end{equation}
The component waves ($p_{1f},p_{2f}$) and ($p_{1b},p_{2b}$) travel spatially through the vessel with the PWV $u$, which under the above assumptions remains constant along the considered vessel segment. Let $L$ denote the length of the vessel segment between the two measurement points $x_1,x_2$. Then both the forward and backward component waves only differ by a time shift (delay) $\tau = L/u,$ which is induced by the PWV $u$. This results in the physical relations
    \begin{equation}\label{eq_delay2}
        p_{2f}(t) = p_{1f}(t-\tau)\,,
        \qquad \text{and} \qquad
        p_{1b}(t) = p_{2b}(t-\tau) \,.
    \end{equation}
Note that functions fulfilling the delay relations \eqref{eq_delay2} are also connected to solutions of one-dimensional advection equations; cf.~\cite{Hubmer_Neubauer_Ramlau_Voss_2018}. Inserting these relations into \eqref{eq_two_waves} we thus obtain the system of equations
    \begin{equation*}
        p_1(t) = p_{1f}(t) + p_{2b}(t-\tau) \,,
        \qquad \text{and}\qquad
        p_2(t) = p_{1f}(t-\tau) + p_{2b}(t) \,,
    \end{equation*}
which after application of the Fourier transform $\mathcal{F},$ where $\hat{p}(\omega):=(\mathcal{F}p)(\omega)$, yields
    \begin{equation}\label{eq_problem}
        \hat{p}_{1}(\omega) = \hat{p}_{1f}(\omega) + \hat{p}_{2b}(\omega)e^{-i\omega L/u} \,,
        \qquad \text{and} \qquad
        \hat{p}_{2}(\omega) = \hat{p}_{1f}(\omega)e^{-i\omega L/u} + \hat{p}_{2b}(\omega)\,.
    \end{equation}
Hence, if the PWV $u$ is known, the component waves $\hat{p}_{1f},\hat{p}_{2b}$ can be found directly via
    \begin{equation}\label{eq_direct}
        \hat{p}_{1f}(\omega) = \frac{\hat{p}_1(\omega) - \hat{p}_2(\omega) e^{-i \omega L/u}}{1-e^{-2 i \omega L/u}}\,,
        \qquad\text{and}\qquad 
        \hat{p}_{2b}(\omega) = \frac{\hat{p}_2(\omega) - \hat{p}_1(\omega) e^{-i \omega L/u}}{1-e^{-2 i \omega L/u}}\,.
    \end{equation}
However, the denominator $1 - e^{-2i\omega L/u}$ in the above formulas is not bounded away from zero for all $\omega$, leading to problems with well-posedness and stability with respect to data noise. A simple remedy for this is to replace the denominator by a filtered version $\lambda_\alpha(1-e^{-2i\omega L/u})$, where $\lambda_\alpha:\C\to\C$ is an approximation of the identity, but bounded away from zero for every \emph{regularization parameter} $\alpha>0$. Possible choices for $\lambda_\alpha$ are e.g., 
    \begin{equation}\label{eq_directreg}
        \lambda_1(x) = x + \alpha\,,
        \quad
        \lambda_2(x) = \sgn (\mathrm{Re}\,x) \max(|\mathrm{Re}\,x|,\alpha)+i\,\sgn(\mathrm{Im}\,x)\max(|\mathrm{Im}\,x|,\alpha)\,,
    \end{equation}
which are reminiscent of Tikhonov-type  and hard thresholding filters, respectively \cite{Engl_Hanke_Neubauer_1996}. However, this straightforward approach has a number of drawbacks:
    \begin{itemize}
        \item First of all, the above filters do not essentially differentiate between high-frequency and low-frequency components $\omega$ in the denominator. This results in a loss of smoothness of the reconstructions and thus numerical instabilities, cf.~\cref{sec:numerics}.
        \item As we will observe numerically in \cref{sec:numerics}, it is beneficial to measure more than two pulse waves $p_1$ and $p_2$ for obtaining unique reconstructions of the PWV $u$. While the above approach can be generalized to this case, the resulting system of equations becomes over-determined, leading to further mathematical challenges.
        \item Finally, and perhaps most importantly, the above approach requires knowledge of the PWV $u$, which is exactly the parameter one wants to estimate in practice.
    \end{itemize}

These drawbacks are the motivation for considering different approaches to the pulse wave splitting and PWV estimation problem. In particular, we reframe it as a parameter estimation problem in operator form commonly encountered in Inverse Problems. Following a theoretical analysis of this reformulated problem, we are able to apply the tools of regularization theory \cite{Engl_Hanke_Neubauer_1996}, and to design stable reconstruction algorithms for obtaining both the component waves and an estimate for the PWV $u$.

The outline of this paper is as follows: In \cref{sec:math_model}, we derive and mathematically analyse an operator formulation of the pulse wave problem, for which we develop reconstruction algorithms in \cref{sec:reconstruction}. These are tested in a number of numerical experiments in \cref{sec:numerics}, which is followed by a discussion and conclusion in \cref{sec:conclusion}.

\section{Mathematical Model}\label{sec:math_model}

In this section, we aim to overcome the challenges discussed in the introduction for solving the pulse wave problem \eqref{eq_problem}. To do so, we first assume that, instead of only two waves $p_1$ and $p_2$, we have access to multiple pulse waves $p_1,\ldots,p_N$ obtained at $N \geq 2$ different measurement points along the considered vascular segment. As before, we assume that each of these waves is given as the sum of a forward and a backward (reflected) wave, i.e.,  
    \begin{equation}\label{eq_summation}
        p_k(t) = p_{kf}(t) + p_{kb}(t)\,,
        \qquad
        \forall \, 1\leq k\leq N \,.
    \end{equation}
Furthermore, we assume that the considered vascular segment is non-branching, and that the PWV $u$ is constant there. Then, as before, the forward and reflected waves perform a steady motion, and thus are delayed versions of each other. This is expressed by the delay equations
    \begin{equation}\label{eq_delay}
        p_{kf}(t) = p_{1f}(t-\tau_k) \,,
        \qquad \text{and} \qquad 
        p_{kb}(t) = p_{Nb}(t-(\tau_N-\tau_k))\,,
        \qquad
        \forall \, 1\leq k\leq N \,,
    \end{equation}
where the delays $\tau_k=L_k/u$ are defined in terms of the lengths $L_k$ of the vessel between the first and $k$-th measurement point along the vessel segment. Together, \eqref{eq_summation} and \eqref{eq_delay} are the constitutive equations of pulse wave splitting and PWV estimation, and are illustrated in \cref{fig:delays}. Combining theses equations, we ultimately obtain the following system of equations:
    \begin{equation}\label{fundamental_time}
        p_k(t) = p_{1f}(t-\tau_k) + p_{Nb}(t-(\tau_N-\tau_k))\,, \qquad \forall \,k=1,\ldots,N \,.
    \end{equation}
    \begin{figure}[ht!]
        \centering
        \includegraphics[width=\textwidth,trim={50 50 200 0},clip]{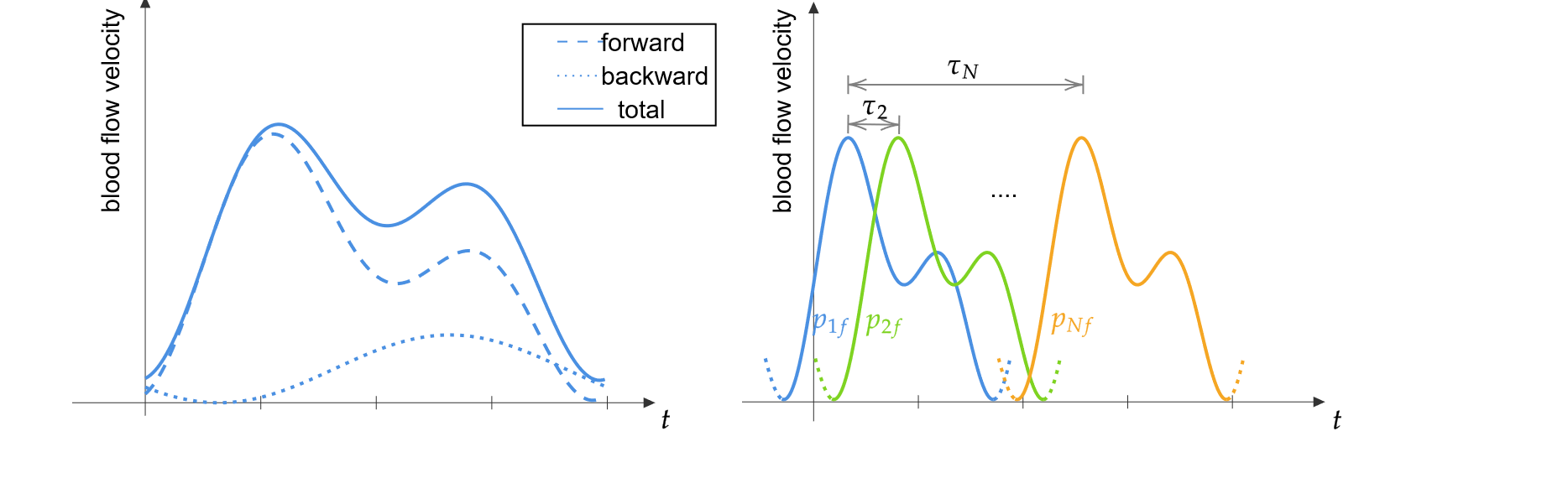}
        \caption{Illustration of summation (left) and delay (right) equations \eqref{eq_summation} and \eqref{eq_delay}.}
        \label{fig:delays}
    \end{figure}

Note that for given measurements $p_1,\ldots,p_N$, solving the system \eqref{fundamental_time} for $p_{1f}$, $p_{Nb}$, and the PWV $u$ is sufficient to determine a complete solution for the problem of pulse wave splitting and PWV estimation, since the forward and reflected waves in other measurement points can then be directly calculated via the delay equations \eqref{eq_delay}. In the following analysis, we require the PWV $u$ to be bounded away from zero, such that induced time-shifts in \eqref{eq_delay} are well defined. This assumption is not restricting the practical relevance of the model, since a PWV $u=0$ would imply no pulse. If we additionally assume a certain smoothness of the pulse waves, i.e., that the waves are elements of some Sobolev space $H^s(\R)$, then the system of equations \eqref{fundamental_time} leads to 

\begin{problem}[Pulse wave splitting and PWV estimation in time domain]\label{prob1}
Let $N\geq2$, $\varepsilon>0$, $r\geq s\geq 0,$ $p_k\in H^s(\R)$ and, for all $k=1,\ldots,N$, let $L_k\in\R^+$ be given and define $\tau_k=L_k/u$. Then the inverse pulse wave splitting and PWV estimation problem in time domain is defined as the reconstruction of $p_{1f},p_{Nb}\in H^r(\R)$ and $u\geq \varepsilon$ via \eqref{fundamental_time} from the data $p_1,\ldots,p_N$.
\end{problem}

In addition to \cref{prob1}, we consider the pulse wave splitting problem, i.e., the splitting of pulse waves into forward and backward parts assuming a known PWV $u$.

\begin{problem}[Pulse wave splitting in time domain]\label{prob2}
Let $N\geq 2$, $u > 0,$ $r\geq s\geq 0,$ and, for all $k=1,\ldots,N,$ let $p_k\in H^s(\R)$ and $L_k\in\R^+$ be given and define $\tau_k=L_k/u$. Then the pulse wave splitting problem in time domain is defined as the reconstruction of $p_{1f},p_{Nb}\in H^r(\R)$ via \eqref{fundamental_time} from the data $p_1,\ldots,p_N$.
\end{problem}

As a subproblem of \cref{prob1}, \cref{prob2} is a linear problem which is both practically interesting in itself \cite{Augl_Thesis_2022}, and will be of help in solving the full nonlinear \cref{prob1}.

\subsection{Problem Formulation as Operator Equation and Analysis}

In our main pulse wave equation system \eqref{fundamental_time}, the time shifts $\tau_k$ appear as function arguments. These shifts are easier to handle in the Fourier domain, where they turn into multiplicative factors. More precisely, applying the Fourier transform to \eqref{fundamental_time} and using the Fourier-Shift Theorem yields
    \begin{equation}\label{fundamental_frequency}
        \hat{p}_k(\omega) = \hat{p}_{1f}(\omega)e^{-i\omega L_k/u} + \hat{p}_{Nb}(\omega)e^{-i\omega (L_N-L_k)/u} \,.
    \end{equation}
 This transformation has a number of advantages: First, the unknown PWV $u$ now appears in a multiplicative factor instead of as a function argument. Furthermore, it is easier to handle function periodicity in the Fourier-domain, and finally, smoothness conditions can be implemented easily in this setting. To simplify the notation below, we define $e_k(\omega,u) := e^{-i\omega L_k/u}$ and $\Tilde{e}_k(\omega,u) := e^{-i\omega (L_N-L_k)/u}.$ For any fixed $\varepsilon>0$ and $\mathcal{U}:=[\varepsilon,\infty)$ we also define the function spaces
    \begin{equation}\label{innerproducts}
        \mathcal{X}_s := (L^2_{s}(\R))^2\times \mathcal{U}, \qquad
        \text{and }\qquad\mathcal{Y}_s := (L^2_{s}(\R))^N,
    \end{equation} 
    where $L^2_s(\R)$ is defined via the norm $\|f\|_{L^2_s(\R)}:=\|(1+|\bullet|^2)^{s/2} f(\bullet)\|_{L^2(\R)}.$ Note that $L^2_r(\R)\subset L^2_s(\R)$ for $r>s,$ $L^2_0(\R)=L^2(\R)$, and $f\in H^s(\R)$ if and only if $\hat{f}\in L^2_s(\R).$ 
    The Hilbert spaces $\mathcal{X}_s$ and $\mathcal{Y}_s$ are equipped with the canonical inner products
    \begin{equation}\label{norms}
    \begin{aligned}
        \spr{(w_1,w_2,u),(x_1,x_2 \,, v)}_{\mathcal{X}_s} &:= \spr{w_1,x_1}_{L^2_{s}(\R)} + \spr{w_2,x_2}_{L^2_{s}(\R)} + uv \,, 
        \\
        \spr{(y_1,\ldots,y_N),(z_1,\ldots,z_N)}_{\mathcal{Y}_s} &:= \sum_{k=1}^N\spr{y_k,z_k}_{L_s^2(\R)} \,.
    \end{aligned}
    \end{equation}

For these weighted Lebesque spaces the following Lemma holds.
\begin{lemma}
For any $r\geq s \geq 0$ and $x \in L^2_r(\R)$ there holds
    \begin{equation}\label{whys1}
        \|x(\omega)\omega^{r-s}\|_{L_s^2(\R)}\leq\|x(\omega)\|_{L^2_r(\R)} \,.
    \end{equation}
\end{lemma}
\begin{proof}
For any fixed $r\geq s \geq 0$ it holds that
    \begin{equation*}
    \begin{aligned}
        \|x(\omega)\omega^{r-s}\|^2_{L_s^2(\R)} &= \int_\R(1+|\omega|^2)^{s} \underbrace{(|\omega|^2)^{r-s}}_{\leq (1+|\omega|^2)^{r-s}}|x(\omega)|^2\,\mathrm{d}\omega\\&\leq
        \int_\R (1+|\omega|^2)^r|x(\omega)|^2 \,\mathrm{d}\omega
        =\|x(\omega)\|^2_{L^2_r(\R)} \,,
    \end{aligned} 
    \end{equation*}
    from which the assertion directly follows.
\end{proof}
Note that \cref{prob1} is equivalent to finding a solution $(\hat{p}_{1f},\hat{p}_{Nb},u)$ of \eqref{fundamental_frequency} given the Fourier-transformed data $(\hat{p}_{1},\ldots,\hat{p}_{N})$. In order to solve this inverse problem, we want to make use of tools from \emph{regularization theory}, and thus we rephrase the problem as an operator equation between suitable Hilbert spaces. To do so, we first need the following definition.
 
\begin{definition}\label{def:F}
Let $N\in\mathbb{N},$ $r\geq s\geq 0,$ fix $ (L_1,\ldots,L_N)\in(\R^+)^N$ and define the following operators:
\begin{alignat}{3}\label{ABmatrixdef}
     V_N: &&\quad L_r^2(\R)&\to \mathcal{Y}_r: \qquad\qquad \qquad\qquad\;\, z(\omega) &&\mapsto(1,\ldots,1)^T z(\omega),  \nonumber \\
     A:&&\mathcal{U}&\to(\mathcal{Y}_r\to\mathcal{Y}_s):\qquad
     A (u): \,\boldsymbol{y}(\omega) &&\mapsto \mathrm{diag}\left(\left(e_k(\omega,u)\right)_{k=1}^N\right)\boldsymbol{y}(\omega), \\
    B:&&\mathcal{U}&\to (\mathcal{Y}_r\to\mathcal{Y}_s):\qquad B (u): \,\boldsymbol{y}(\omega) &&\mapsto\mathrm{diag}\left(\left(\Tilde{e}_k(\omega,u)\right)_{k=1}^N\right)\boldsymbol{y}(\omega)\, . \nonumber
\end{alignat}
Then, for $\boldsymbol{x}=(x_1,x_2)\in (L^2_{r}(\R))^2$ the operator $F:\mathcal{X}_r\to \mathcal{Y}_s$ is defined as 
    \begin{equation}\label{eq_def:F}
        F(\boldsymbol{x},u):= A(u)V_Nx_1 + B(u)V_Nx_2\,.
    \end{equation}
\end{definition}

Using \cref{def:F}, we can now formulate the system of equations \eqref{fundamental_frequency} as
    \begin{equation}\label{Operator_eq}
        F(\hat{p}_{1f},\hat{p}_{Nb},u) = (\hat{p}_k)_{k=1}^N \,,
    \end{equation}
i.e., as a single operator equation, which leads to the following 

\begin{problem}[Operator form of pulse wave splitting and PWV estimation in frequency domain]\label{prob3} 
Let $N\geq 2,$ $r\geq s\geq 0$ and, for all $k=1,\ldots,N$, let $L_k\in\R^+$ be given. Furthermore, let $F$ be defined as in \eqref{eq_def:F} and assume that $(\hat{p}_k)_{k=1}^N\in \mathcal{Y}_s$ is given. Then, the pulse wave splitting and PWV estimation problem is defined as finding $(\hat{p}_{1f},\hat{p}_{Nb},u)\in \mathcal{X}_r$ such that the nonlinear operator equation \eqref{Operator_eq} holds.
\end{problem}

Concerning the well-definedness of $F$, we obtain the following
\begin{theorem}
The operator $F:\mathcal{X}_r\to \mathcal{Y}_s$ as in \eqref{eq_def:F} is well-defined and satisfies
    \begin{equation*}
        \|F(\boldsymbol{x},u)\|_{\mathcal{Y}_s}\leq \sqrt{2N} \|(x_1,x_2,u)\|_{\mathcal{X}_r} \,,
        \qquad
        \forall \, (\boldsymbol{x},u) = (x_1,x_2,u)\in \mathcal{X}_r \,.
    \end{equation*}
\end{theorem}
\begin{proof}
First, let $u \in \mathcal{U}$ be arbitrary but fixed. Since $|e_k(\omega)|=1$, we can estimate
 \begin{equation}\label{ABnorm}
 \begin{aligned}
     \|A(u)\|_{\mathcal{Y}_r\to \mathcal{Y}_s}
     &\overset{\phantom{\eqref{norms}}}{=} \sup_{\boldsymbol{y}\in\mathcal{Y}_r}\frac{\|A(u)\boldsymbol{y}\|_{\mathcal{Y}_s}}{\|\boldsymbol{y}\|_{\mathcal{Y}_r}}
     =\sup_{\boldsymbol{y}\in\mathcal{Y}_r}\frac{\|(e_ky_k)_{k=1}^N\|_{\mathcal{Y}_s}}{\|\boldsymbol{y}\|_{\mathcal{Y}_r}}\\
     &\overset{\eqref{norms}}{=} \sup_{\boldsymbol{y}\in\mathcal{Y}_r}\frac{\left(\sum_{k=1}^N\|e_ky_k\|^2_{L^2_s(\R)}\right)^{1/2}}{\|\boldsymbol{y}\|_{\mathcal{Y}_r}}= \sup_{\boldsymbol{y}\in\mathcal{Y}_r}\frac{\|\boldsymbol{y}\|_{\mathcal{Y}_s}}{\|\boldsymbol{y}\|_{\mathcal{Y}_r}}\leq 1\, ,
 \end{aligned} 
\end{equation}
where we used $\|\boldsymbol{y}\|_{\mathcal{Y}_s}\leq\|\boldsymbol{y}\|_{\mathcal{Y}_r}$ for $r\geq s\geq0$ in the last step. Analogously, one can also show that $\|B(u)\|_{\mathcal{Y}_r\to\mathcal{Y}_s}\leq 1.$
Furthermore, note that for all $x_i\in L_s^2(\R)$ there holds
\begin{equation}\label{eq_Vbound}
    \|V_N x_i\|_{\mathcal{Y}_s}=\left(\sum_{k=1}^N\|x_i\|^2_{L_s^2(\R)}\right)^{1/2}=\sqrt{N}\|x_i\|_{L_s^2(\R)}\,.
\end{equation}
Now let $(\boldsymbol{x},u)\in \mathcal{X}_r.$ Due to the Cauchy-Schwarz inequality, for all $a,b\in\R$ there holds $a+b\leq (2(a^2+b^2))^{1/2}$. Hence, we can derive the estimate
\begin{equation*}
    \begin{aligned}
        \|F(\boldsymbol{x},u)\|_{\mathcal{Y}_s} &\overset{\phantom{\eqref{ABnorm},\eqref{eq_Vbound}}}{=} \|A(u)V_Nx_1+B(u)V_Nx_2\|_{\mathcal{Y}_s}
        \\&\overset{\phantom{\eqref{ABnorm},\eqref{eq_Vbound}}}{\leq}
        \|A(u)V_Nx_1\|_{\mathcal{Y}_s}+\|B(u)V_Nx_2\|_{\mathcal{Y}_s}
        \\&\overset{\phantom{\eqref{ABnorm},\eqref{eq_Vbound}}}{\leq}
        \|A(u)\|_{\mathcal{Y}_r\to\mathcal{Y}_s}\|V_Nx_1\|_{\mathcal{Y}_s}+\|B(u)\|_{\mathcal{Y}_r\to\mathcal{Y}_s}\|V_Nx_2\|_{\mathcal{Y}_s}
        \\&\overset{\eqref{ABnorm},\eqref{eq_Vbound}}{\leq} 
        \sqrt{N}\left(\|x_1\|_{L_s^2(\R)}+\|x_2\|_{L_s^2(\R)} \right)
        \\&\overset{\phantom{\eqref{ABnorm},\eqref{eq_Vbound}}}{\leq} \sqrt{2N}\left(\|x_1\|^2_{L_s^2(\R)}+\|x_2\|^2_{L_s^2(\R)} \right)^{1/2}
        \leq\sqrt{2N} \|(x_1,x_2,u)\|_{\mathcal{X}_r} \,,
    \end{aligned}
    \end{equation*} 
which yields the assertion.
\end{proof}

In \cref{def:F}, the operator $F$ is defined between $(L^2_r(\R))^2\times \mathcal{U}$ and $((L_s^2(\R)))^N$. Translated into the time-domain, this is equivalent to considering the unknown pulse waves $(p_{1f},p_{Nb})$ as elements of $(H^r(\R))^2$, and the data $(p_k)_{k=1}^N$ as elements of $(H^s(\R))^N$, respectively. Since the Fourier transform is an isomorphism on $L^2(\R)\supset H^s(\R)$, it thus follows that \cref{prob3} is equivalent to \cref{prob1}. Next, we show that for certain values of $s$ and $r$, the operator $F$ is continuously Fr\'echet differentiable. For this, we first need the following

\begin{lemma}\label{Lemma 2.7}
Let $r,s\geq0$ be such that $r\geq s+2.$
    Then the operators $A,B:\mathcal{U}\to (\mathcal{Y}_r\to\mathcal{Y}_s)$ given as in \eqref{ABmatrixdef} 
   are continuously Fr\'echet differentiable with
   \begin{equation}\label{A_derivative}
   \begin{aligned}
              A'(u)h: \boldsymbol{y}(\omega)&\mapsto\mathrm{diag}\left(\frac{i\omega L_1h}{u^2}e_1(\omega,u),\ldots,\frac{i\omega L_1h}{u^2}e_N(\omega,u)\right)\boldsymbol{y}(\omega)\,,\\
              B'(u)h: \boldsymbol{y}(\omega)&\mapsto\mathrm{diag}\left(\frac{i\omega L_1h}{u^2}\Tilde{e}_1(\omega,u),\ldots,\frac{i\omega L_1h}{u^2}\Tilde{e}_N(\omega,u)\right)\boldsymbol{y}(\omega)\,.\\         
   \end{aligned}
   \end{equation}
\end{lemma}
\begin{proof}
    See \cref{appendixA}.
\end{proof}

\begin{theorem}\label{thm:frechet}
Let $r,s\geq0$ be such that $r\geq s+2.$ Then the operator $F:\mathcal{X}_r\to \mathcal{Y}_s$ as defined in \eqref{eq_def:F} is continuous and continuously Fr\'echet differentiable for each $(\boldsymbol{x},u)\in\mathcal{X}_r$, with
    \begin{equation}\label{eq_DefFrechet}
        F'(\boldsymbol{x},u)\boldsymbol{h}= A(u)V_Nh_1+B(u)V_Nh_2+A'(u)(V_nx_1)h_3+B'(u)(V_Nx_2)h_3\, ,
    \end{equation}
where $\boldsymbol{h}=(h_1,h_2,h_3)\in \mathcal{X}_r.$
\end{theorem}
\begin{proof}
Recall first that by the definition \eqref{eq_def:F} of $F$ there holds
    \begin{equation*}
        F(\boldsymbol{x},u) = A(u)V_Nx_1 + B(u)V_Nx_2\,.
    \end{equation*}
Due to \eqref{eq_Vbound} the linear operator $V_N$ is bounded, and thus also continuously Fr\'echet differentiable. Furthermore, from \cref{Lemma 2.7} we also know that both the operators $A$ and $B$ are continuously Fr\'echet differentiable. Now let $D_{x_1}$, $D_{x_2}$, and $D_u$ denote the (partial) Fr\'echet derivatives with respect to the variables $x_1$, $x_2$, and $u$. Then the product rule of Fr\'echet derivatives, cf.~\cite[Theorem~8.9.2]{Dieudonne1969}, yields that $F$ is continuously Fr\'echet differentiable with
    \begin{equation*}
        F'(\boldsymbol{x},u) \boldsymbol{h}
        = 
        D_{x_1} F(\boldsymbol{x},u) h_1
        + 
        D_{x_2} F(\boldsymbol{x},u) h_2
        +
        D_{u} F(\boldsymbol{x},u) h_3 \,.
    \end{equation*}
From the chain rule and the linearity of the operator $V_N$ we obtain
    \begin{equation*}
    \begin{split}
        D_{x_1} F(\boldsymbol{x},u) h_1
        &=
        D_{x_1} (A(u)V_N x_1)h_1 = A(u) D_{x_1} (V_N x_1)h_1 = A(u)V_Nh_1 \,,
        \\
        D_{x_2} F(\boldsymbol{x},u) h_2
        &=
        D_{x_2} (B(u)V_N x_2)h_2 = B(u) D_{x_2}(V_N x_2)h_2 = B(u)V_Nh_2 \,,
        \\
        D_{u} F(\boldsymbol{x},u) h_3 
        &=
        D_u \kl{ A(u)V_Nx_1 + B(u)V_Nx_2 }h_3
        =
        A'(u)(V_Nx_1) h_3 + B'(u)(V_Nx_2) h_3
        \,,
    \end{split}
    \end{equation*}
which yields the assertion.
\end{proof}

To compute the adjoint of the Fr\'echet derivative \eqref{eq_DefFrechet}, we first need the following
\begin{proposition}
Let $r,s\geq0$ be such that $r\geq s+2$, $u\in\mathcal{U}$, and let $A(u),B(u):\mathcal{Y}_r\to \mathcal{Y}_s$ be defined as in \eqref{ABmatrixdef}. Then, their adjoints $A(u)^\ast,B(u)^\ast:\mathcal{Y}_s\to \mathcal{Y}_r$ are given by
\begin{equation}\label{eq:adjoint_AB}
\begin{aligned}
   &A(u)^\ast\boldsymbol{y}(\omega)=\mathrm{diag}\left(\left(E_{r,s}^\ast e_k(\omega,u)^{-1}\right)_{k=1}^N\right)\boldsymbol{y}(\omega),\\
   &B(u)^\ast\boldsymbol{y}(\omega)=\mathrm{diag}\left(\left(E_{r,s}^\ast \Tilde{e}_k(\omega,u)^{-1}\right)_{k=1}^N\right)\boldsymbol{y}(\omega),
\end{aligned}
\end{equation}
where $E_{r,s}^\ast$ denotes the adjoint of the embedding operator $E_{r,s}:L^2_r(\R)\to L^2_s(\R):\; E_{r,s} x=x.$
\end{proposition}
\begin{proof}
Let $\boldsymbol{y}=(y_1,\ldots,y_N)\in \mathcal{Y}_r$ and $\boldsymbol{z}=(z_1,\ldots,z_N)\in \mathcal{Y}_s,$ then
 \begin{equation*}
     \begin{aligned}
          \spr{A(u)\boldsymbol{y},\boldsymbol{z}}_{\mathcal{Y}_s}
          &=\sum_{k=1}^N\spr{e_k(\bullet,u)y_k,z_k}_{L^2_s(\R)}\\
          &=\sum_{k=1}^N\spr{E_{r,s}y_k,\overline{e_k(\bullet,u)} z_k}_{L^2_s(\R)},\\
          &=\sum_{k=1}^N\spr{y_k,E_{r,s}^\ast e_k(\bullet,u)^{-1}z_k}_{L^2_r(\R)}=\spr{A(u)^\ast \boldsymbol{y},\boldsymbol{z}},
     \end{aligned}
 \end{equation*}  
 where we used $\overline{e_k(\bullet,u)}=e_k(\bullet,u)^{-1}.$
The result for $B(u)^\ast$ follows analogously.
\end{proof}
\begin{proposition}
    Let $\boldsymbol{y}=(y_k)_{k=1}^N\in \mathcal{Y}_r$. The adjoint of the operator $V_N:L^2_r(\R)\to \mathcal{Y}_r$ is given by $V_N^\ast \boldsymbol{y} =\sum_{k=1}^Ny_k.$ 
\end{proposition}
\begin{proof}
    Let $x\in L^2_r(\R),\,\boldsymbol{y}\in \mathcal{Y}_r,$ then
\begin{equation*}
    \spr{V_N x,\boldsymbol{y}}=\sum_{k=1}^N\spr{x,y_k}_{L^2_r(\R)}=\spr{x,\sum_{k=1}^Ny_k}_{L^2_r(\R)},
\end{equation*}
which yields $V_N^\ast \boldsymbol{y} =\sum_{k=1}^Ny_k.$ 
\end{proof}
Now we combine the above results into the following
\begin{theorem}\label{thm:adjoint_frechet}
Let $r,s\geq0$ be such that $r\geq s+2,$ $(\boldsymbol{x},u)=(x_1,x_2,u)\in \mathcal{X}_r$, $\boldsymbol{y}=(y_1,\ldots,y_N)\in \mathcal{Y}_s$, and let $F'(\boldsymbol{x},u):\mathcal{X}_r\to \mathcal{Y}_s$ be defined as in \eqref{eq_DefFrechet}. Then, its adjoint $F'(\boldsymbol{x},u)^\ast:\mathcal{Y}_s\to \mathcal{X}_r$ is given by
\begin{equation}\label{eq:adjoint_frechet}
    F'(\boldsymbol{x},u)^\ast\boldsymbol{y} = \left(\begin{array}{cc}
   V_N^\ast A(u)^\ast \boldsymbol{y} 
    \\
   V_N^\ast B(u)^\ast \boldsymbol{y} 
    \\
    \spr{A'(u)V_N(x_1)+B'(u)V_N(x_2),\boldsymbol{y}}_{\mathcal{Y}_s}
    \end{array}\right)\, .
\end{equation}
\end{theorem}
\begin{proof}
For $\boldsymbol{h}=(h_1,h_2,h_3)\in\mathcal{X}_r$ we can directly calculate
\begin{equation*}
    \begin{aligned}
        \spr{F'(\boldsymbol{x},u)\boldsymbol{h},\boldsymbol{y}}_{\mathcal{Y}_s}&=
        \spr{A(u)V_Nh_1,\boldsymbol{y}}_{\mathcal{Y}_s}+
        \spr{B(u)V_Nh_2,\boldsymbol{y}}_{\mathcal{Y}_s}\\
        &\quad+\spr{A'(u)(V_Nx_1)h_3+B'(u)(V_Nx_2)h_3,\boldsymbol{y}}_{\mathcal{Y}_s}\\
        &=
        \spr{h_1,V_N^\ast A(u)^\ast \boldsymbol{y}}_{L_r^2(\R)}+\spr{h_2,V_N^\ast B(u)^\ast \boldsymbol{y}}_{L_r^2(\R)}\\
        &\quad+h_3\spr{A'(u)V_Nx_1+B'(u)V_Nx_2,\boldsymbol{y}}_{\mathcal{Y}_s}\, ,
    \end{aligned}
\end{equation*}
which yields the assertion.
\end{proof}

Note that the adjoint embedding operator $E^\ast_{r,s}$ is also implicitly appearing in the operator $F'(\boldsymbol{x},u)^\ast$ as a component of $A(u)^\ast$, which leads us to the following remark.
\begin{remark}
    The adjoint embedding operator can be explicitly computed via
    \begin{equation*}
        \begin{aligned}
            \langle\, E_{r,s} {x},{y}\,\rangle_{L_s^2(\R)}
            &=\langle\,{x},{y}\,\rangle_{L_s^2(\R)}=\int_\R (1+|\omega|^2)^s{x}(\omega){y}(\omega)\,\mathrm{d}\omega
            \\&=\int_\R (1+|\omega|^2)^r{x}(\omega)(1+|\omega|^2)^{s-r}{y}(\omega),\mathrm{d}\omega
            \\&=\spr{{x},(1+|\bullet|^2)^{s-r}y}_{L_r^2(\R)}
            =:\spr{{x},E_{r,s}^\ast {y}}_{L_r^2(\R)} \,, 
        \end{aligned}
    \end{equation*}
    for ${x}\in L^2_r(\R)$ and ${y}\in L_s^2(\R)$, which implies that $E_{r,s}^\ast {y}(\omega)=(1+|\omega|^2)^{s-r} {y}(\omega).$
    Using this, we can write the adjoint of the Fr\'echet derivative of $F$ given in \eqref{eq:adjoint_frechet} as
    \begin{equation*}
   F'(\boldsymbol{x},u)^\ast\boldsymbol{y}(\omega) = \sum_{k=1}^N\left(\begin{array}{cc}
    (1+|\omega|^2)^{s-r} y_k(\omega) e_k(\omega,u)^{-1}  
    \\
    (1+|\omega|^2)^{s-r} y_k(\omega) \Tilde{e}_k(\omega,u)^{-1} 
    \\
    \spr{x_1\frac{i\bullet L_k}{u^2}e_k(\bullet,u),y_k}_{L_s^2(\R)}+\spr{x_2\frac{i\bullet (L_N-L_k)}{u^2}\Tilde{e}_k(\bullet,u),y_k}_{L_s^2(\R)}
    \end{array}\right)\, .
\end{equation*}
\end{remark}

After analyzing the nonlinear operator $F$ describing the pulse wave splitting and PWV estimation problem simultaneously, we now derive similar results for an operator describing only the linear pulse wave splitting \cref{prob2}, which is defined in the following 

\begin{definition}\label{def:F_u}
Let $r\geq s\geq0,$ $N\in\mathbb{N}$, $u > 0$, and $L_k\in\mathbb{R}^+$ for $k=1,\ldots,N$ be given. Then for $\boldsymbol{x}=(x_1,x_2)\in (L^2_r(\R))^2$ the operator $F_u:(L^2_r(\R))^2 \to \mathcal{Y}_s$ is defined as $F_u(\boldsymbol{x}):=F(\boldsymbol{x},u),$ with $F$ as in \cref{def:F}.
\end{definition}

With this definition, the linear problem of (only) pulse wave splitting resulting from \eqref{fundamental_frequency}, i.e., \cref{prob2} can now be stated in operator form as
    \begin{equation}\label{lin_prob}
        F_u(\hat{p}_{1f},\hat{p}_{Nb})=(\hat{p}_k)_{k=1}^N,
    \end{equation}
which allows it to be formalized into the following

\begin{problem}[Operator form of pulse wave splitting in frequency domain]\label{prob4}
Let $r\geq s\geq0,$ $N\geq 2$, $u\geq\varepsilon>0$ and, for all $k=1,\ldots,N$, let $L_k\in\R^+$ be given. Furthermore, let $F_u$ be as in \cref{def:F_u} and let $(\hat{p}_k)_{k=1}^N\in \mathcal{Y}_s$. Then the pulse wave splitting problem is defined as finding $(\hat{p}_{1f},\hat{p}_{Nb})\in (L^2_r(\R))^2$ such that the linear operator equation \eqref{lin_prob}
holds.
\end{problem}

Similarly to above, we now obtain the adjoint of $F_u$ in the subsequent

\begin{corollary}\label{cor:adjoint}
Let $r\geq s\geq0$ and let $F_u:(L^2_{r}(\R))^2 \to \mathcal{Y}_s$ be as in \cref{def:F_u}.  Then, for $\boldsymbol{y}=(y_1,\ldots,y_N)$ its adjoint operator $F_u^\ast:\mathcal{Y}_s \to (L^2_{r}(\R))^2$ is given by 
 \begin{equation*}
    F_u^\ast\boldsymbol{y} = \left(\begin{array}{cc}
     V_N^\ast A(u)^\ast \boldsymbol{y} 
    \\
    V_N^\ast B(u)^\ast \boldsymbol{y} 
    \end{array}\right)\,.
\end{equation*}
\end{corollary}
\begin{proof}
This is a consequence of \cref{thm:adjoint_frechet} and the fact that $F_u$ is a restriction of $F$ defined in \eqref{eq_def:F} onto its arguments $(x_1,x_2)$.
\end{proof}
\begin{remark}
    Even though the existence of the Fr\'echet derivative of the nonlinear operator $F$ requires a smoothness assumption $r\geq s+2,$ no such assumption is needed in the linear case involving the operator $F_u$, i.e., the smoothness requirement stems from the parameter $u$. 
\end{remark}
The operator $F$ given in \eqref{eq_def:F} acts linearly on the functions $x_1,x_2$, and nonlinearly on the PWV $u$. This motivates our later development of solution methods that alternate between these variables. For this, it is convenient to also define the restriction of $F$ on only the PWV $u$ in the next definition.

\begin{definition}\label{def:F_x}
Let $r\geq s\geq0,\,N\in\mathbb{N},\,\boldsymbol{x}=(x_1,x_2)\in(L^2_{r}(\R))^2$ and $L_k\in\mathbb{R}^+$ for $k=1,\ldots,N$ be given. Then, for $u\in\mathcal{U}$ the operator $F_{\boldsymbol{x}}:\mathcal{U} \to \mathcal{Y}_s$ is defined as 
$F_{\boldsymbol{x}} (u)(\omega):=F(\boldsymbol{x},u).$
\end{definition}

The Fr\'echet derivative of $F_{\boldsymbol{x}}$ and its adjoint are computed in the following 

\begin{corollary}
Let $r,s\geq 0$ be such that $r\geq s+2$ and $\boldsymbol{x}=(x_1,x_2)\in (L^2_r(\R))^2$. Then, for $F_{\boldsymbol{x}}:\mathcal{U}\to \mathcal{Y}_s$ as in \cref{def:F_x} the Fr\'echet derivative $F_{\boldsymbol{x}}'(u):\mathcal{U}\to \mathcal{Y}_s$ is given by
\begin{equation*}
F'_{\boldsymbol{x}}(u)h= A'(u)(V_nx_1)h+B'(u)(V_Nx_2)h\, .
\end{equation*}
Its adjoint $F_{\boldsymbol{x}}'(u)^\ast:\mathcal{Y}_s\to \mathcal{U}$ is given by 
\begin{equation*}
    F_{\boldsymbol{x}}'(u)^\ast \boldsymbol{y} = \spr{A'(u)V_Nx_1+B'(u)V_Nx_2,\boldsymbol{y}}_{\mathcal{Y}_s} \,.
\end{equation*}
\end{corollary}
\begin{proof}
The assertion follows from \cref{thm:frechet} and \cref{thm:adjoint_frechet} and the fact that $F_{\boldsymbol{x}}$ is a restriction of $F$ defined in \eqref{eq_def:F} onto its argument $u$.
\end{proof}
\section{Reconstruction Approaches}\label{sec:reconstruction}

In this section, we introduce the specific solution methods we use to solve \cref{prob3,prob4}. First, consider the operator equation
    \begin{equation}\label{opeq}
        G(x) = y \,,
    \end{equation}
 where $G:X\to Y$ is a continuous, Fr\'echet-differentiable, nonlinear operator between two Hilbert spaces $X,Y$. The general goal in inverse problems is to reconstruct (an approximation of a) solution $x^\ast$ of \eqref{opeq} from a noisy version $y^\delta$ of the true data $y$.
Such problems are considered as ill-posed if they lack existence or uniqueness of a solution, or most importantly, stablility with respect to data noise.
 As a remedy for these stabilization issues, one needs to apply \emph{regularization methods}, such as \emph{Tikhonov regularization} or \emph{Landweber iteration} \cite{Engl_Hanke_Neubauer_1996}. In Tikhonov regularization, the regularized solution minimizes the functional 
    \begin{equation*}
        \mathcal{T}_\alpha^\delta (x) := \|G(x)-y^\delta\|^2+\alpha\|x-x_0\|^2 \,,
    \end{equation*}
 where $\alpha>0$ is a suitably chosen regularization parameter. Under appropriate assumptions on $G$ and $\alpha$, it can be shown that the minimizers of $\mathcal{T}_\alpha^\delta$ converge subsequentially to a minimum norm solution of $G(x)=y$ as $\delta\to 0$ \cite{Engl_Hanke_Neubauer_1996}. In general, no directly computable solution is available for the minimizer of the Tikhonov functional. In particular, this means that iterative optimization methods are applied for minimizing $\mathcal{T}_\alpha^\delta$. Alternatively, iterative algorithms can be applied directly to \eqref{opeq}, such as Landweber iteration \cite{Kaltenbacher_Neubauer_Scherzer_2008}.
In order to function as regularization methods, iterative schemes have to be stopped appropriately, the stopping index then acting as a regularization parameter.
For detailed studies on nonlinear-ill-posed problems and regularization theory we refer to \cite{Engl_Hanke_Neubauer_1996, Kaltenbacher_Neubauer_Scherzer_2008,Louis_1989}.

\subsection{Linear Tikhonov Regularization}\label{sec:linTikh}

First, we consider Tikhonov regularization for the linear \cref{prob4}, i.e., for the operator equation \eqref{lin_prob}. With $\boldsymbol{\hat{p}}^\delta:=(\hat{p}^\delta)_{k=1}^N$ and the initial guess $\boldsymbol{x}^0=0$, Tikhonov regularization here takes the form
    \begin{equation}\label{eq_helper_01}
        \boldsymbol{x}_\alpha^\delta = \arg\min_{\boldsymbol{x}\in (L^2_r(\R))^2}\|F_u(\boldsymbol{x})-\boldsymbol{\hat{p}}^\delta\|_{\mathcal{Y}_s}^2+\alpha\|\boldsymbol{x}\|_{(L^2_r(\R))^2}^2 \,.     
    \end{equation}
The minimizer of \eqref{eq_helper_01} can be calculated explicitly via
    \begin{equation}\label{linTikh}
        \boldsymbol{x}_\alpha^\delta = (F_u^\ast F_u+\alpha I)^{-1}F_u^\ast \boldsymbol{\hat{p}}^\delta  \,,
    \end{equation}
which gives rise to the following

\begin{method}[linTikh]\label{method:linTikh}
For given data $\boldsymbol{\hat{p}}^\delta=(\hat{p}_k^\delta)_{k=1}^N$ and a known PWV $u$, the \emph{linTikh} method computes $\boldsymbol{x}_\alpha^\delta$ as the minimizer of the linear Tikhonov functional \eqref{eq_helper_01} i.e., via \eqref{linTikh},
where $\alpha$ is a suitably chosen regularization parameter.
\end{method}
The solution computed by Tikhonov regularization converges to the best approximate solution of $F_u(\boldsymbol{x})=\hat{\boldsymbol{p}}$ as $\delta\to 0,$ if $\alpha$ is chosen appropriately in dependence on $\delta$ \cite{Engl_Hanke_Neubauer_1996}.

Note that besides Tikhonov regularization, for $N=2$ also the direct approach defined by \eqref{eq_direct} and \eqref{eq_directreg} can be used to solve the linear problem \eqref{lin_prob}. A numerical comparison of the Tikhonov solution compared to the direct approach is shown in \cref{sec:numerics}.

\subsection{Minimization of the Linear Tikhonov Functional}\label{sec:minTikh}

The nonlinearity of \cref{prob3} stems only from the unknown PWV $u$, i.e., from a single scalar parameter. Hence, we propose to make use of the fact that for a fixed $u$, \cref{prob3} reduces to the linear \cref{prob4}. For this, we leverage the fact that from a physical perspective, the value of the PWV is known to approximately lie in the range of $u\in[1\, \text{m/s},10\, \text{m/s}]$. In this rather small range of admissible values for the PWV, we can repeatedly solve the linear subproblem $F_u(\boldsymbol{x})=\boldsymbol{\hat{p}}^\delta$ via the linTikh \cref{method:linTikh} for a finite set $U\subset\mathcal{U}$ of admissible PWVs. The solution of \cref{prob3} is then found within the set of solutions to the linear problems, such that the residual norm is minimized. This approach is summarized in the following

\begin{method}[minTikh]\label{method:minTikh}
For given data $\boldsymbol{\hat{p}}^\delta=(\hat{p}_k^\delta)_{k=1}^N$ let $U=(u_1,\ldots,u_K)$ be a finite set of admissible parameters for the PWV $u$. The \emph{minimizing Tikhonov} method consists of computing an approximation $u_{\alpha}^\delta$ of the PWV via
    \begin{equation}\label{minTikh}
    \begin{aligned}
        \boldsymbol{x}_{\alpha}^\delta(u)&=(F_u^\ast F_u+\alpha I)^{-1}F_u^\ast \boldsymbol{\hat{p}}^\delta \,, \qquad \forall \, u\in U\,,
        \\
        u_\alpha^\delta&=\arg\min_{u\in U} \|F_u(\boldsymbol{x}_{\alpha}^\delta(u))-\boldsymbol{\hat{p}}^\delta\|_{\mathcal{Y}_s} \,,
    \end{aligned}
    \end{equation}
where $\alpha$ is a suitable regularization parameter. The corresponding approximate solution regarding pulse wave splitting is given by $\boldsymbol{x}_{\alpha}^\delta(u_{\alpha}^\delta)$.
\end{method}
The minTikh method performs a search for the optimal PWV $u \in U$ with respect to the (converging) solutions of the linear Tikhonov functionals, i.e., a best approximate solution within $U$ for the PWV is found.

\subsection{Alternate Direction Method}

As described above, a popular approach for solving nonlinear ill-posed problems is by iterative methods, for example via Landweber iteration. However, due to the system structure of the operator $F$ in \eqref{eq_def:F}, it is beneficial to iteratively approach the solution in alternate steps, switching between the variables on which the problem depends linearly and nonlinearly, respectively. This approach is motivated by coordinate descent methods \cite{Buccini_Donatelli_Ramlau_2018,Ortega_Rheinboldt_2000,Wright2015} or the well-known ADMM (Alternate Direction Method of Multipliers) algorithm \cite{Bleyer_Ramlau_2013,Bleyer_Ramlau_2015,Boyd_Parikh_chu_peleatp_Eckstein_2011,Hong2016} from optimization theory. A benefit of these approaches is that for each of the alternating directions the step-size can be chosen separately and thus more effectively. 

The idea of this approach is to split up the minimization problem 
    \begin{equation*}
        (\boldsymbol{x}_\alpha^\delta,u_\alpha^\delta) = \arg\min_{(\boldsymbol{x},u)\in \mathcal{X}_s}\|F(\boldsymbol{x},u)-\boldsymbol{\hat{p}}^\delta\|^2+\alpha\|\boldsymbol{x}-\boldsymbol{x}^0\|^2 \,,
    \end{equation*}
into two subproblems that can be addressed separately and iteratively.

\begin{method}[ADM]\label{method:adm}
For given data $\boldsymbol{\hat{p}}^\delta=(\hat{p}_k^\delta)_{k=1}^N$, an initial guess $(\boldsymbol{x}^0,u^0)$, and a suitable regularization parameter $\alpha$, the \emph{ADM} method is defined via the iterative procedure
    \begin{equation}\label{adm}
    \begin{aligned}
        \boldsymbol{x}^{k} &= \arg\min_{\boldsymbol{x}\in (L^2_r(\R))^2}\|F_{u^{k-1}}(\boldsymbol{x})-\boldsymbol{\hat{p}}^\delta\|^{2}_{\mathcal{Y}}+\alpha\|\boldsymbol{x}-\boldsymbol{x}^0\|^{2}_{(L^2_r(\R))^2} \,, 
        \\
        u^k&=\arg\min_{u\in U}\|F_{\boldsymbol{x}^k}(u)-\boldsymbol{\hat{p}}^\delta\|^{2}_{\mathcal{Y}_s} \,,
    \end{aligned}
    \end{equation}
for $1\leq k \leq k^\ast$, where $k^\ast$ denotes a suitable stopping index. The approximate solution obtained by the ADM method is given by $(\boldsymbol{x}^{k^\ast},u^{k^\ast})$.
\end{method}

The first minimization problem in \eqref{adm} corresponds to the pulse wave splitting \cref{prob4}, and can for example be solved using \cref{linTikh}. The second minimization problem is a nonlinear optimization problem, for which we employ a gradient descent algorithm or Landweber iteration, respectively, defined by
\begin{equation}\label{landweber}
        u_{k+1} = u_k+\omega_k F_{\boldsymbol{x}}'(u_k)^\ast(\boldsymbol{\hat{p}}^\delta-F_{\boldsymbol{x}}(u_k)) \,,
    \end{equation}
where $x_0^\delta=x_0$ denotes an initial guess and $\omega_k$ is a scaling parameter or \emph{stepsize} \cite{Engl_Hanke_Neubauer_1996,Neubauer_2017,Scherzer_1995}.
We use the \emph{steepest descent} stepsize, given by
    \begin{equation*}
        \omega_k = \frac{\|s_k\|^2}{\|F_{\boldsymbol{x}}'(u_k) s_k\|^2}\,,
        \qquad
        s_k := F_{\boldsymbol{x}}'(u_k)^\ast (\boldsymbol{\hat{p}}^\delta-F_{\boldsymbol{x}}(u_k)) \,.
    \end{equation*} 
The choice of suitable stopping indices is discussed in the next section.

Even though we are not aware of convergence results which cover our specific setting - where one variable direction follows a descent algorithm and the other one solves an infinite-dimensional linear inverse problem - our proposed algorithm can be seen as a modified 2-block Gauß-Seidel method, whose convergence to a stationary point in finite dimensions has been shown in \cite{Grippof_1999}.
From a numerical point of view, we observe local convergence, see \cref{sec:numerics}.
\section{Numerical Experiments}\label{sec:numerics}

In this section, we describe the discretization of the problem, its implementation, and present reconstruction results for simulated data. For the discretization of the problem, we use $m$ equidistant function evaluations in both time and frequency space. If not specified otherwise, we use $m=500$ in all experiments. All computations have been performed in MATLAB 2022b. The Fourier transform is approximated by the DFT/FFT, and for the evaluation of the FFT the in-built Matlab function \texttt{fft} is used. In all numerical experiments we choose $s=0,$ i.e., we consider the operator $F:\mathcal{X}_r\to (L^2(\R))^N$, since this puts the least smoothness requirements on the (noisy) data $\boldsymbol{\hat{p}}^\delta$. 

Throughout this chapter, noisy data $\boldsymbol{\hat{p}}^\delta=(\hat{p}_k)_{k=1}^N$ fulfills
    \begin{equation*}
        \frac{\|\hat{p}_k-\hat{p}_k^\delta\|}{\|\hat{p}_k\|} = \delta \,, \qquad k=1,\ldots,N,
    \end{equation*}
where $\delta$ denotes the \emph{relative noise level}. However, when dealing with real data pulse waves, the relative noise level is highly dependent on the specific data acquisition method, and is difficult to estimate. In related acquisition methods, a noise level of 5-10\% can realistically be expected, which we use for our simulation setting.
For an approximate solution $(x_1,x_2,u),$ we consider the two error measures
    \begin{equation}\label{error}
        e_{\text{fit}} := \frac{\|(x_1,x_2)-(\hat{p}_{1f},\hat{p}_{Nb})\|_{(L^2_r(\R))^2}}{\|(\hat{p}_{1f},\hat{p}_{Nb})\|_{(L^2_r(\R))^2}}\,,
        \qquad 
        e_{\text{res}}:=\frac{\|F(x_1,x_2,u)-\boldsymbol{\hat{p}}^\delta\|_{\mathcal{Y}_0}}{\|\boldsymbol{\hat{p}}^\delta\|_{\mathcal{Y}_0}}\,,
    \end{equation}
which we call \emph{relative fitting error} and \emph{relative residual error}, respectively. The relative fitting error is a quantity only available if the true solution $(\hat{p}_{1f},\hat{p}_{Nb})$ of the problem $F(\hat{p}_{1f},\hat{p}_{Nb},u)=\boldsymbol{\hat{p}}$ is known, i.e., for simulated data only. As regularization parameter choice rules, we use a-priori guesses for Tikhonov regularization and a fixed maximal number of steps for Landweber iteration. To create a test-function that resembles realistic conditions, we create the backwards travelling wave from the forward wave via several reflective points, each with a different distance and attenuation coefficient, while the forward wave is given by
    $p_{1f}=\sin^3\left(2\pi t\right)+0.3\sin^2\left(2\pi t+1\right),\; t\in[0,1], $
 shifted to zero mean. The distance of the first to last measurement point is always chosen as $L_N=0.15\,\text{m}$, corresponding to a realistic total vessel length in the human brain. For the 3-point experiment, the middle point is at $L_2=0.09\,\text{m}$ and for the 5-point experiment the middle points are given as $(L_2,L_3,L_4)=(0.04\,\text{m},0.09\,\text{m},0.12\,\text{m})$. The duration of one cardiac cycle, chosen as 0.75s, corresponds to a frequency of 80 heart beats per minute. The test-functions $(p_{kf},p_{kb})_{k=1}^N$ as well as the resulting noisy test-data $\boldsymbol{\hat{p}}^\delta$ are depicted in \cref{fig:data}. While the simulation is performed in the frequency domain according to \cref{prob3,prob4}, all figures depict the waves in time-domain, i.e., after inverse Fourier transformation of the obtained solutions.

\begin{figure}[ht!]
    \centering
    \includegraphics[width=.49\textwidth,trim={0 0 20 20},clip]{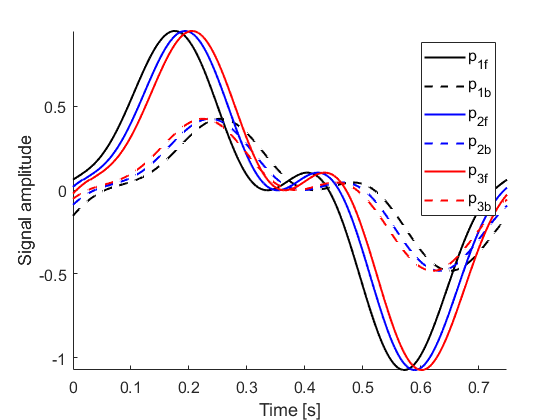}
    \includegraphics[width=.49\textwidth,trim={0 0 20 20},clip]{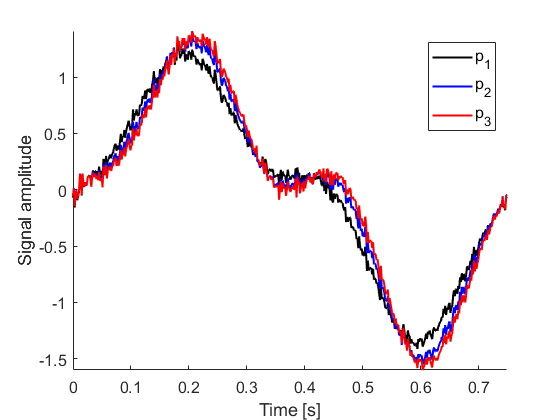}
    \caption{Simulation data for $N=3$. Split waves $p_{1f},p_{2f},p_{3f},p_{1b},p_{2b},p_{3b}$ (left) and total waves $p_1,p_2,p_3$ at each measurement point with added noise level 5\% (right).}
    \label{fig:data}
\end{figure}

\subsection{Numerical Results}

For the first test, we consider the linear problem of pulse wave splitting with a known PWV $u$, i.e., \cref{prob4}. \cref{fig:linearprob} shows the true solution compared to the solution computed by the linTikh \cref{method:linTikh} for several different regularization parameters $\alpha$ and smoothing parameters $r$. Even though we required $r\geq s+2$ in \cref{sec:math_model} for differentiability of the underlying operator $F$, \cref{fig:linearprob} shows satisfactory smoothing already for $r=1$. Hence, in the subsequent tests we always consider $r=1$. In \cref{fig:linearprob}, we further see that without smoothing (top row), the high frequency oscillations remain even for a large regularization parameter. However, for suitably chosen parameters $(\alpha,r)$ as in d) and e), we observe a satisfactory approximation of the true solution.

\begin{figure}[ht!]
    \centering
    \begin{tabular}{c}
    \subfloat[\centering $\alpha=1.2\cdot 10^{-5},\,r=0$]{\includegraphics[width=0.323\textwidth,trim={10 0 30 20},clip]{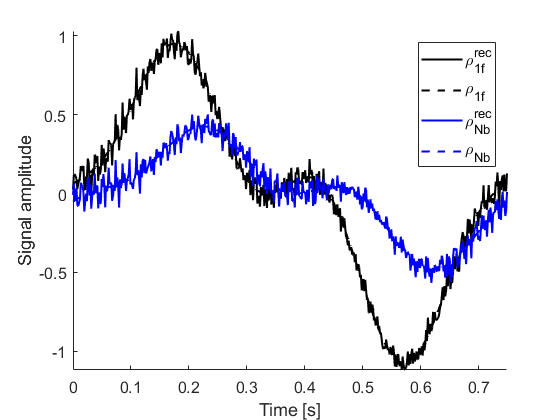}}  \label{fig:smallalphasol}%
    \subfloat[\centering $\alpha=1.2\cdot 10^{-3},\,r=0$]{\includegraphics[width=0.323\textwidth,trim={10 0 30 20},clip]{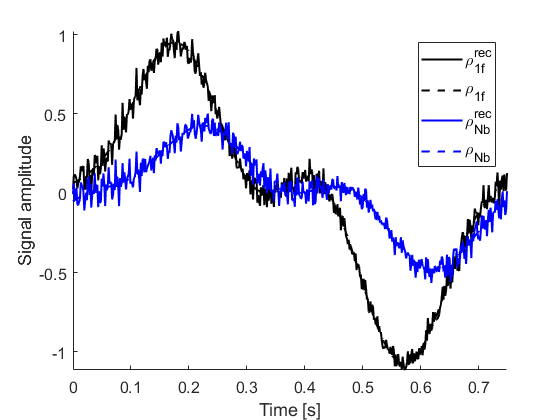}}  \label{fig:optialphasol}%
    \subfloat[\centering $\alpha=1.2\cdot 10^{-1},\,r=0$]{\includegraphics[width=0.323\textwidth,trim={10 0 30 20},clip]{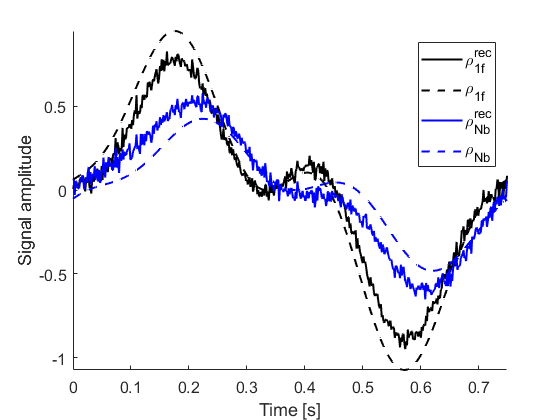}} \label{fig:largealphasol}%
    \\
    \subfloat[\centering $\alpha=1.2\cdot 10^{-5},\,r=1$]{\includegraphics[width=0.323\textwidth,trim={10 0 30 20},clip]{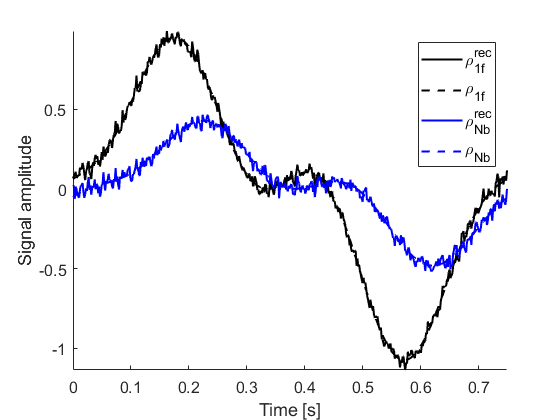}}  \label{fig:nonsmoothed_sol}%
    \subfloat[\centering $\alpha=1.2\cdot 10^{-3},\,r=1$]{\includegraphics[width=0.323\textwidth,trim={10 0 30 20},clip]{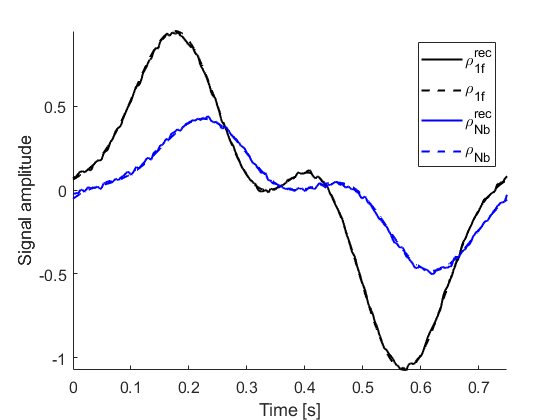}}  \label{fig:smoothed_sol1}%
    \subfloat[\centering $\alpha=1.2\cdot 10^{-1},\,r=1$]{\includegraphics[width=0.323\textwidth,trim={10 0 30 20},clip]{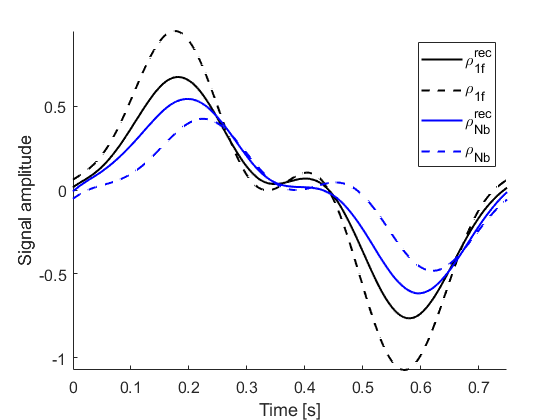}} \label{fig:smoothed_sol2}%
    \end{tabular}
    \caption{Results of the linTikh method \eqref{linTikh} for varying regularization parameter $\alpha$ and smoothing parameter $r$. The relative noise level $\delta=5\%$ and the number of measurement points $N=2$ are fixed in all graphs.}
    \label{fig:linearprob}
\end{figure}

Next, we present results for the different solution-methods for \cref{prob3}. In the left plot of \cref{fig:impracticalmethods}, the relative residual error of the minTikh \cref{method:minTikh} for $N=2$ and $N=3$ measurement points, with a true PWV $u=5\, \text{m/s},$ minimized over the finite set $U=\{1+9(k-1)/99|k=1,\ldots,100\}$ is shown. This choice of $U$ implies that we look for a PWV in the interval between 1 m/s and 10 m/s. Furthermore, the relative residual error of the direct approach \eqref{eq_directreg} with hard thresholding is computed for each $u\in U$ as well. We see that for the minTikh method for $N=2,$ as well as for the direct approach, there is no unique minimizer of the residual error close to the true PWV. Only for the minTikh method for $N=3$ a minimum is obtained close to the actual PWV. We conclude that for unique solvability of \cref{prob3}, at least $N\geq 3$ is required. \cref{fig:impracticalmethods} also shows the impact of the regularization parameter on the reconstruction quality and reconstructed PWV for the minTikh method (top right plot) and ADM \cref{method:adm} (bottom plots), respectively. In particular for the ADM method, we can observe the importance of a proper choice of regularization parameter and initial value. Note that the dependence on the regularization parameter increases significantly for an unsuitable initial guess (bottom left plot, $u^0=10 \, \text{m/s}$). However, for an initial guess close to the solution (bottom right plot, $u^0=4 \, \text{m/s})$), we obtain a large interval of regularization parameters for which the PWV is reconstructed properly, which also generally holds for the minTikh method. In particular, the PWV-reconstruction is robust with respect to the regularization parameter, which means that in practical situations, the regularization parameter can be tuned manually, such that the reconstructed pulse waves still retain their physical characteristics, while noise-related oscillations are reduced.

\begin{figure}[ht!]
    \centering
    \includegraphics[width=.48\textwidth,trim={20 0 10 10},clip]{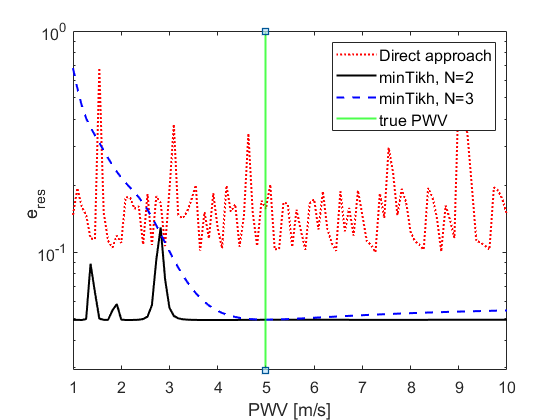}
    \includegraphics[width=.48\textwidth,trim={20 0 10 10},clip]{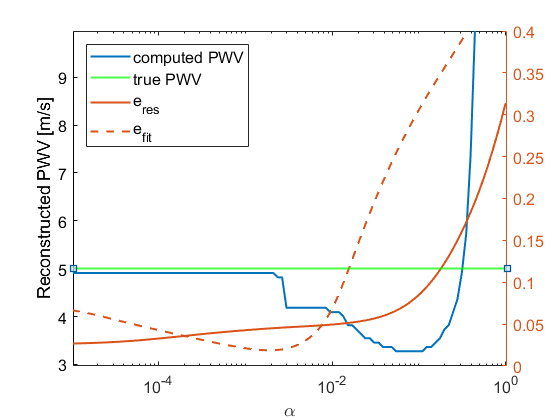}
    \\
    \includegraphics[width=.48\textwidth,trim={20 0 10 10},clip]{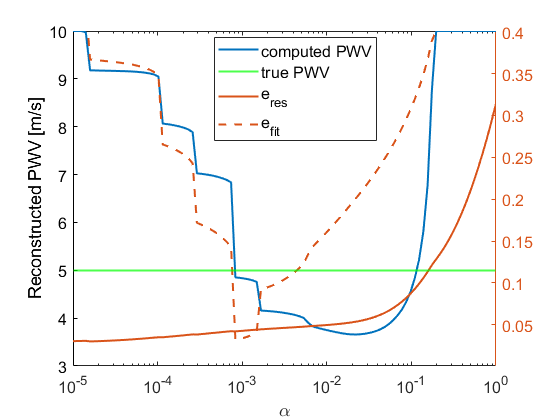}
    \includegraphics[width=.48\textwidth,trim={20 0 10 10},clip]{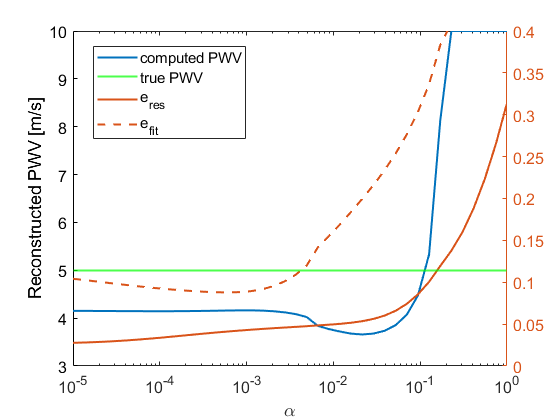}
    \caption{Quality measure $e_{res}$ for the minTikh method \eqref{minTikh} with varying number of measurement points, compared to the direct regularization approach with hard thresholding \eqref{eq_direct},\eqref{eq_directreg} (top, left). Influence of the regularization parameter on reconstruction quality and reconstructed PWV for the minTikh method \eqref{minTikh} (top, right) and ADM method \eqref{adm} with initial value $u^0=10 \, \text{m/s}$ (bottom, left) or $u^0=4\, \text{m/s}$ (bottom, right). All experiments have a relative noise level of $\delta=5\%$, and all but the top left have a fixed number of $N=3$ measurement points.}
    \label{fig:impracticalmethods}
\end{figure}

In \cref{fig:minTikh0}, reconstruction results for the minTikh \cref{method:minTikh} in the exact (noise-free) case are depicted, using $\alpha=10^{-5}$ as a-priori regularization parameter. The method is able to reconstruct the pulse wave forms and the actual PWV precisely, such that the dashed and full lines in the plot actually overlap, making the dashed line hardly visible. For the experiments with $5\%$ noise level, the regularization parameter is chosen as $\alpha=10^{-3}$. Here, we find that for the true PWV $u=2 \,\text{m/s}$ a clear minimum of the residual error is obtained, and the actual PWV is reconstructed almost accurately. However, in the case $u=8 \,\text{m/s}$ the reconstructed PWV $u$ deviates much more from the actual solution, and the difference of magnitude in the relative residual error for different PWVs is visually hardly noticeable for large PWVs. This effect occurs because the difference between the measurement waves in several points is only very small for large PWVs, which can be seen in the right column of \cref{fig:minTikh0.05}. 
One remedy for this situation would be to increase the lengths $L_k$ of the considered vessel segment, since this would increase the corresponding delays $\tau_k=L_k/u$ in Equation~\eqref{fundamental_time} in an inversely proportional manner to a decrease in the PWV $u$. We also deduce that the addition of more measured points (in between the first and last one) only leads to a minor improvement, even though it increases the available data points and computational costs significantly. Results for the ADM \cref{method:adm} are depicted in \cref{fig:ADM0,fig:ADM0.05}, for analogous experiments as for the minTikh method. 
In particular, we observe convergence to a specific PWV after a sufficient number of outer iterations. However, due to potential local minima, the computed PWV may not actually be the true one, underlining the need of a proper initial guess. 
Since we observe a converging behaviour for the calculated PWV $u$, we implement a stopping rule which terminates if no relevant changes appear for a certain amount of iterations, or more precisely, the stopping index $k^\ast\geq10$ is chosen as the first index $k$ such that 
    \begin{equation*}
        \sum_{j=1}^{10}|u^{k-j+1}-u^{k-j}|<10^{-3} \,.
    \end{equation*}
The inner iteration is stopped analogously, but with a different threshold of $10^{-4}$.

In Figure~\ref{fig:noise_dependence}, error values and reconstructed PWVs for optimal regularization parameters are depicted in dependence on the noise level. Again, we observe that in particular the minTikh method is able to closely reconstruct the true PWV even for very high noise levels. In contrast, reconstruction accuracy for the ADM method decreases already for 2\% noise.

All visual findings are supported by the error values in \cref{table:errortable}, where also the number of linTikh evaluations is shown. This value is an indicator for the efficiency and computational cost of the method, since the application of the linTikh \cref{method:linTikh} is the most computational expensive step in both the minTikh and ADM method, respectively. For minTikh, the number of linTikh evaluations is equal to $|U|$, while in the ADM method, it is equal to the number of total outer iterations $k^\ast$ as in the description of \cref{method:adm}. The computational cost in both methods can be regulated by adjusting $|U|$ and the stepsize $\omega_k$ in \eqref{landweber}, respectively. 

In summary, both proposed methods minTikh and ADM solve pulse wave splitting and PWV estimation very well in the simulated data setting, specifically for lower PWVs. The minTikh method has the advantage that no initial guess for the PWV is required, and that the PWV is accurately reconstructed for a very large range of regularization parameters.
In contrast, the ADM method is strongly dependent on the initial guess for the PWV and is less reliable with respect to the choice of the regularization parameter. Therefore, we generally recommend the minTikh method for pulse wave velocity estimation. 

\begin{figure}[ht!]
    \centering
    \begin{tabular}{c}
    \includegraphics[width=0.318\textwidth,trim={5 0 30 20},clip]{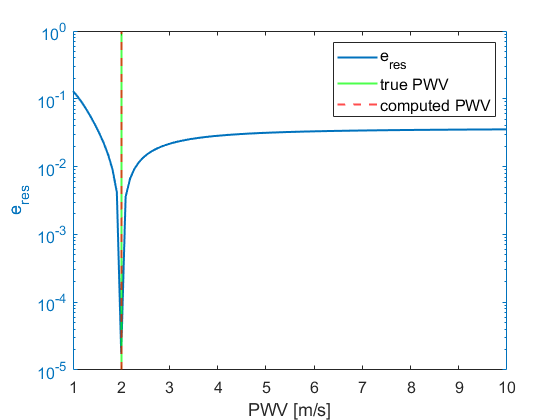} %
    \includegraphics[width=0.318\textwidth,trim={10 0 30 20},clip]{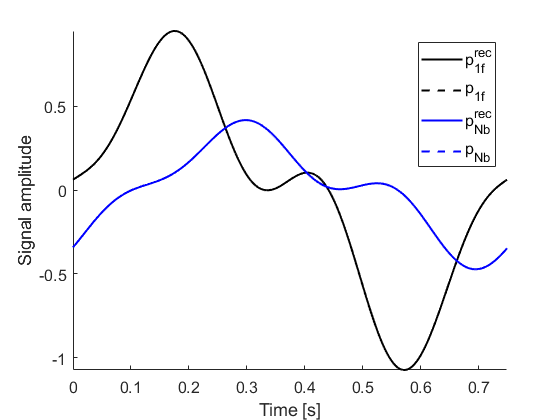} %
    \includegraphics[width=0.318\textwidth,trim={10 0 30 20},clip]{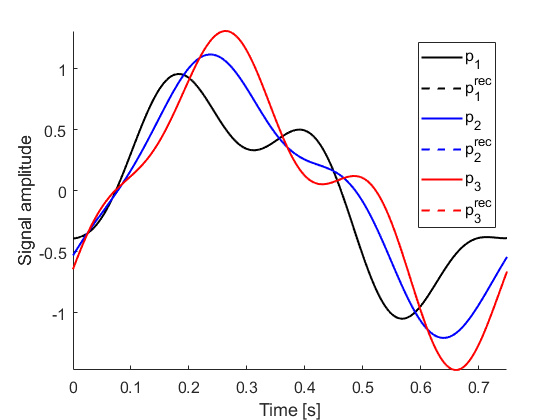}
    \\
    \includegraphics[width=0.318\textwidth,trim={5 0 30 20},clip]{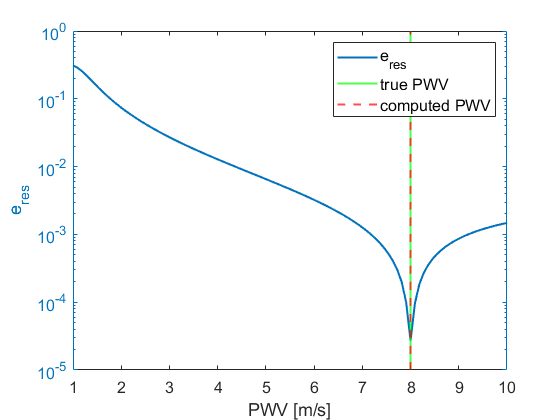}
    \includegraphics[width=0.318\textwidth,trim={10 0 30 20},clip]{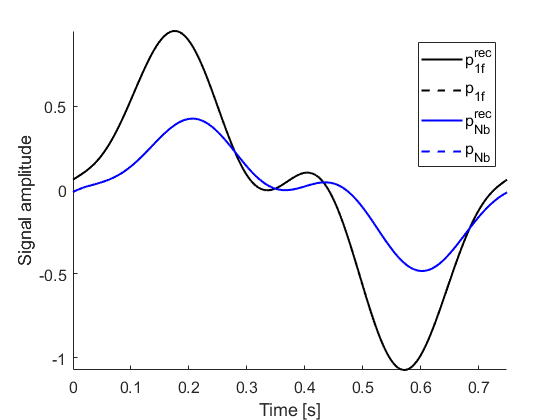}
    \includegraphics[width=0.318\textwidth,trim={10 0 30 20},clip]{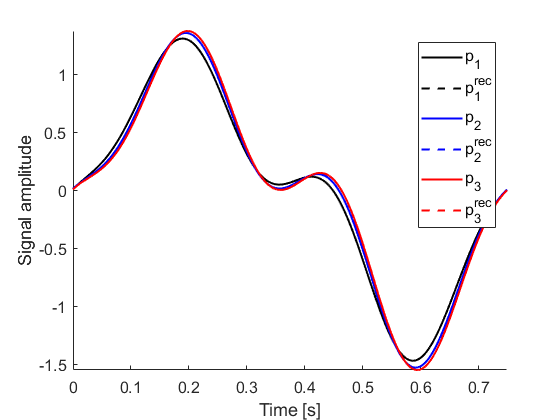}
    \end{tabular}
    \caption{Results for the minTikh \cref{minTikh} without noise for true velocities $u^\ast=2$ (top) and $u^\ast=8$ (bottom): Relative residual error for every $u\in U$ (left), the reconstructed wave forms $(p^{\text{rec}}_{1f},p^{\text{rec}}_{Nb})=\mathcal{F}^{-1}(x_{1,\alpha},x_{2,\alpha})$ compared to actual wave forms $(p_{1f},p_{Nb})$ (middle) and the reconstructed data waves $\boldsymbol{p}^{\text{rec}}=\mathcal{F}^{-1}F(\boldsymbol{x}_\alpha,u_\alpha)$ compared to the simulated data $\boldsymbol{p}$ (right).}
    \label{fig:minTikh0}
\end{figure}

\begin{figure}[ht!]
    \centering
    \begin{tabular}{c}
    \includegraphics[width=0.318\textwidth,trim={5 0 30 20},clip]{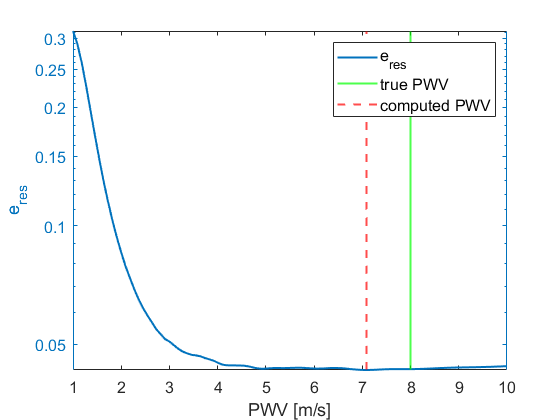} %
    \includegraphics[width=0.318\textwidth,trim={10 0 30 20},clip]{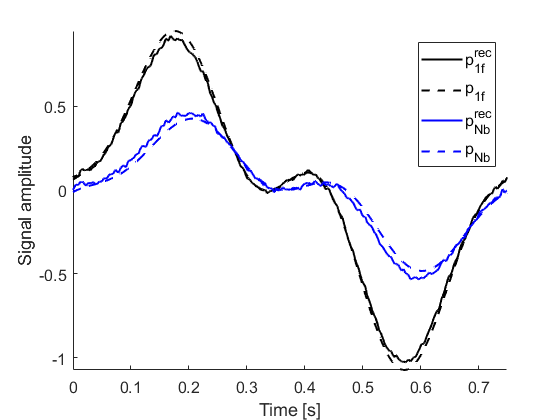} %
    \includegraphics[width=0.318\textwidth,trim={10 0 30 20},clip]{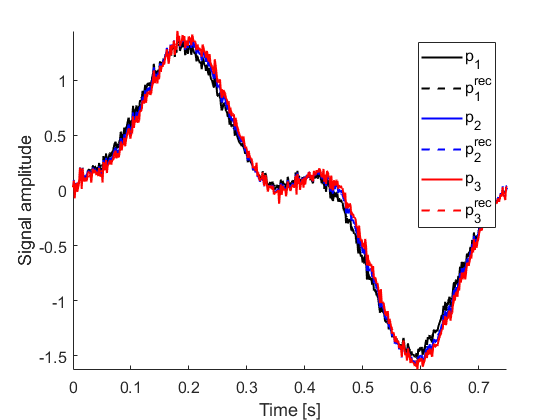}  
    \\
    \includegraphics[width=0.318\textwidth,trim={5 0 30 20},clip]{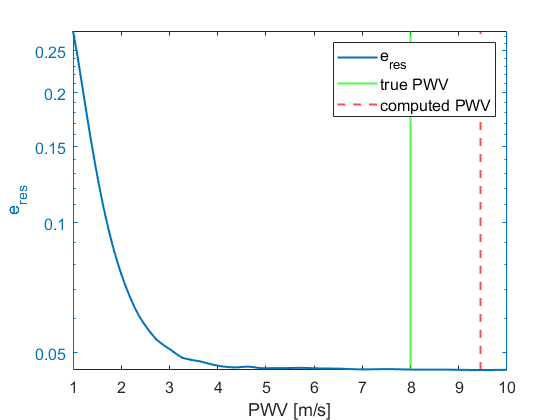} 
    \includegraphics[width=0.318\textwidth,trim={10 0 30 20},clip]{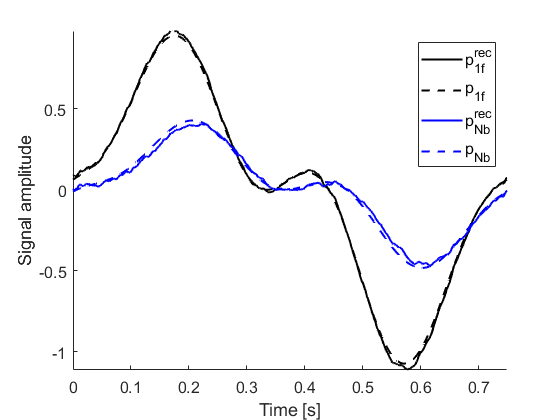}  
    \includegraphics[width=0.318\textwidth,trim={10 0 30 20},clip]{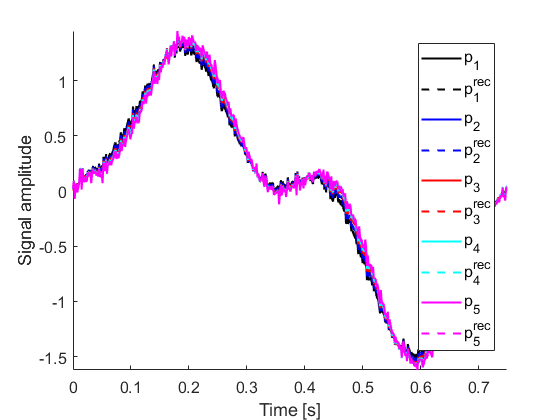}  
    \end{tabular}
    \caption{Results for the minTikh \cref{method:minTikh} with relative noise level $\delta=5\%$ and $u^\ast=8$, for $N=3$ (top) and $N=5$ (bottom): Relative residual error for every $u\in U$ (left), the reconstructed wave forms $(p^{rec}_{1f},p^{rec}_{Nb})=\mathcal{F}^{-1}(x_{1,\alpha}^\delta,x_{2,\alpha}^\delta)$ compared to actual wave forms $(p_{1f},p_{Nb})$ (middle) and the reconstructed data waves $\boldsymbol{p}^{\text{rec}}=\mathcal{F}^{-1}F(\boldsymbol{x}_\alpha^\delta,u_\alpha^\delta)$ compared to the simulated data $\boldsymbol{p}$ (right).}
    \label{fig:minTikh0.05}
\end{figure}

\begin{figure}[ht!]
    \centering
    \begin{tabular}{c}
    \includegraphics[width=0.318\textwidth,trim={10 0 30 20},clip]{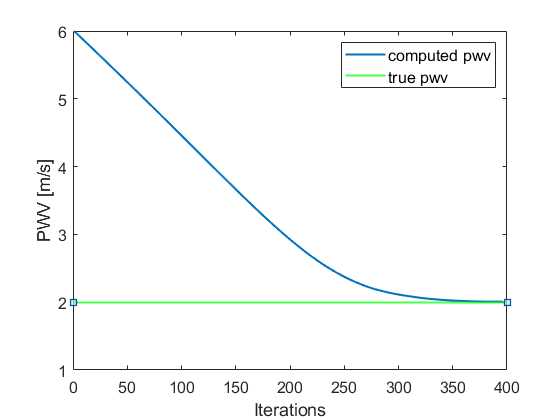} %
    \includegraphics[width=0.318\textwidth,trim={10 0 30 20},clip]{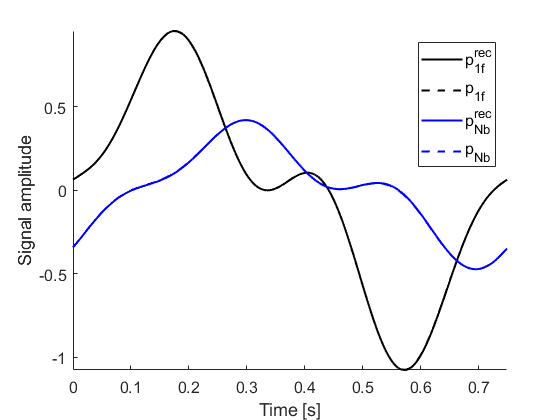} %
    \includegraphics[width=0.318\textwidth,trim={10 0 30 20},clip]{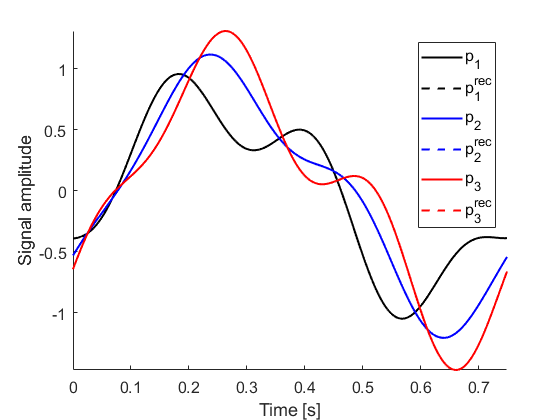}  
    \\
    \includegraphics[width=0.318\textwidth,trim={10 0 30 20},clip]{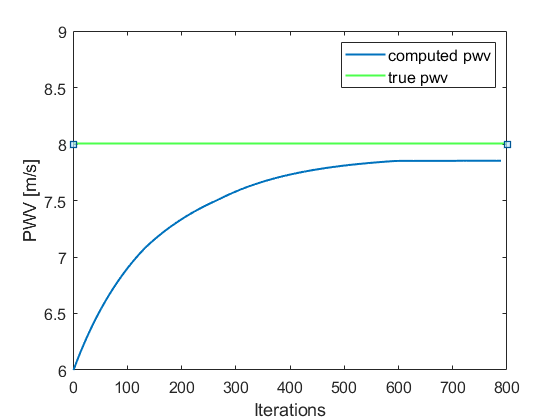} 
    \includegraphics[width=0.318\textwidth,trim={10 0 30 20},clip]{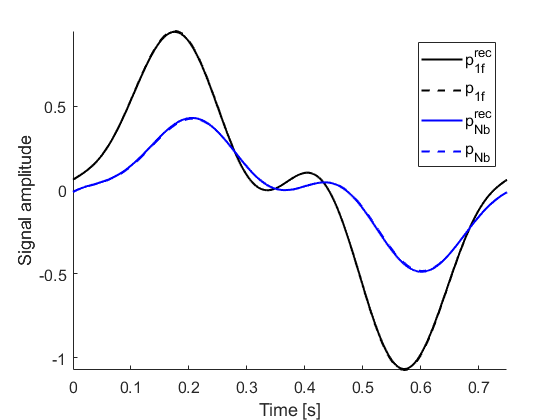}  
    \includegraphics[width=0.318\textwidth,trim={10 0 30 20},clip]{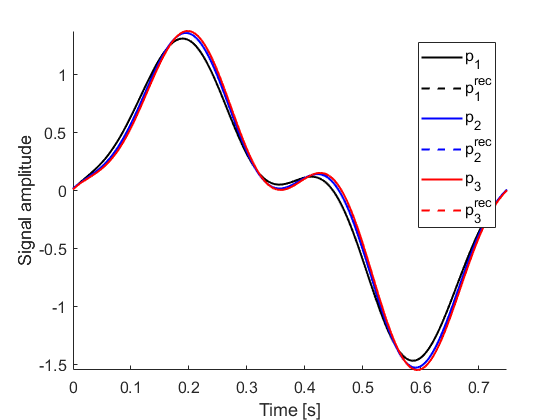}  
    \end{tabular}
    \caption{Results for the ADM Method~\eqref{method:adm} without noise for true velocities $u=2$ (top) and $u=8$ (bottom): The reconstructed PWV in dependence of the outer iteration number (left), the reconstructed wave forms $(p^{rec}_{1f},p^{rec}_{Nb})=\mathcal{F}^{-1}(x_1^{k^\ast},x_2^{k^\ast})$ compared to actual wave forms $(p_{1f},p_{Nb})$ (middle) and the reconstructed data waves $\boldsymbol{p}^{\text{rec}}=\mathcal{F}^{-1}F(\boldsymbol{x}^{k^\ast},u^{k^\ast})$ compared to the simulated data $\boldsymbol{p}$ (right).}
    \label{fig:ADM0}
\end{figure}

\begin{figure}[ht!]
    \centering
    \begin{tabular}{c}
    \includegraphics[width=0.318\textwidth,trim={10 0 30 20},clip]{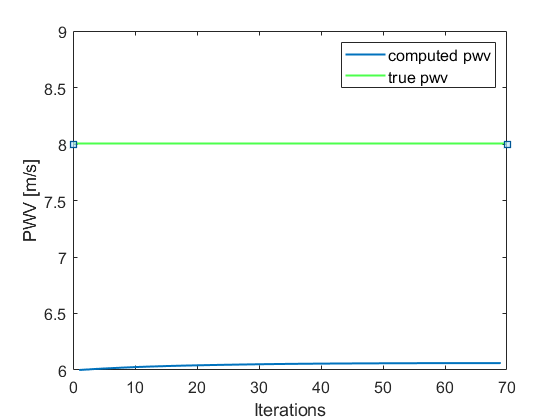} %
    \includegraphics[width=0.318\textwidth,trim={10 0 30 20},clip]{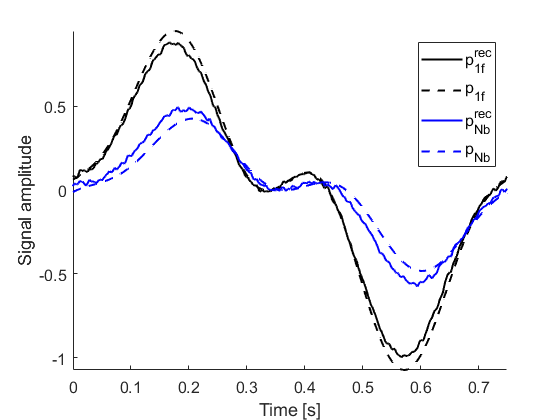} %
    \includegraphics[width=0.318\textwidth,trim={10 0 30 20},clip]{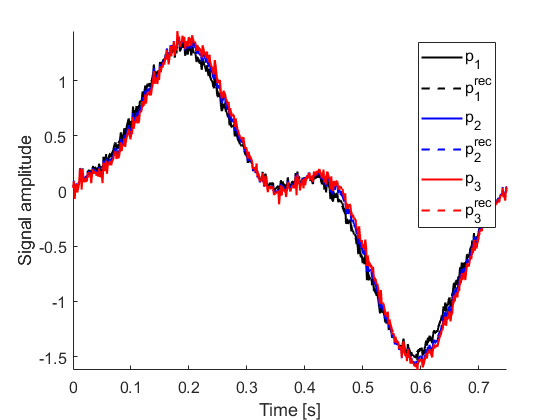}  
    \\
    \includegraphics[width=0.318\textwidth,trim={10 0 30 20},clip]{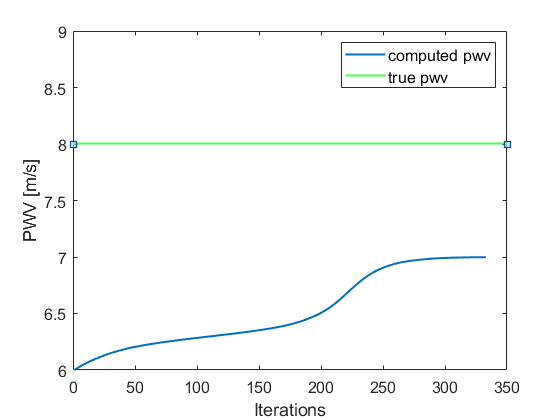} 
    \includegraphics[width=0.318\textwidth,trim={10 0 30 20},clip]{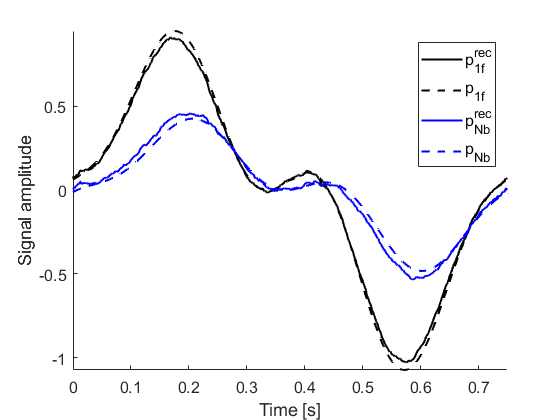}  
    \includegraphics[width=0.318\textwidth,trim={10 0 30 20},clip]{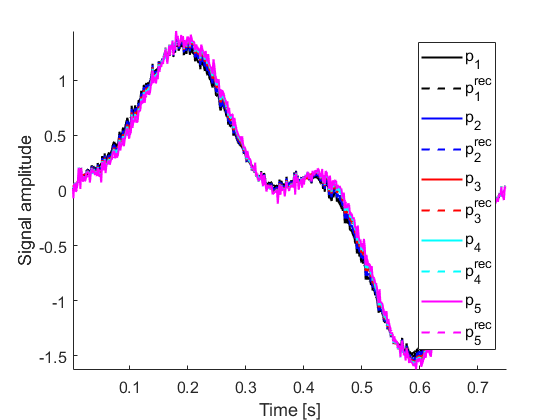}  
    \end{tabular}
    \caption{Results for the ADM \cref{method:adm} with relative noise level $\delta=5\%$ for $N=3$ (top) and $N=5$ (bottom): The reconstructed PWV in dependence of the outer iteration number (left), the reconstructed wave forms $(p^{rec}_{1f},p^{rec}_{Nb})=\mathcal{F}^{-1}(x_1^{k^\ast},x_2^{k^\ast})$ compared to actual wave forms $(p_{1f},p_{Nb})$ (middle) and the reconstructed data waves $\boldsymbol{p}^{\text{rec}}=\mathcal{F}^{-1}F(\boldsymbol{x}^{k^\ast},u^{k^\ast})$ compared to the simulated data $\boldsymbol{p}$ (right).}
    \label{fig:ADM0.05}
\end{figure}

\begin{figure}[ht!]
    \centering
    \includegraphics[width=.48\textwidth,trim={20 0 0 10},clip]{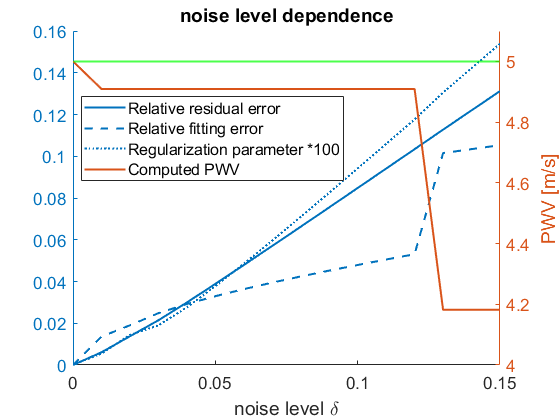}
    \includegraphics[width=.48\textwidth,trim={20 0 0 10},clip]{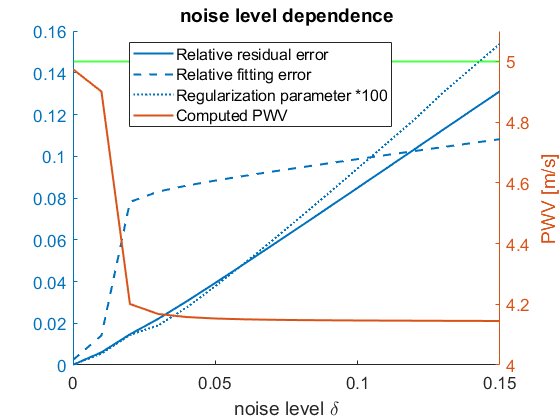}
    \caption{Impact of noise level onto reconstruction quality for the minTikh method (left) and ADM method with initial guess $u^0=4\, \text{m/s}$ (right) for $N=3.$ The green line depicts the true PWV.}
    \label{fig:noise_dependence}
\end{figure}

\begin{table}[ht!]
    \centering
    \caption{Reconstruction results in different settings. The row ``linTikh eval.'' indicates the number of evaluations of the linTikh \cref{method:linTikh} in the corresponding experiments.}
    \scalebox{0.85}{
    \begin{tabular}{|l|l|l|l|l|l|l|l|l|l|l|}
    \cline {4-11}
    \multicolumn{3}{c|}{} & \multicolumn{4}{c|}{exact data} &   \multicolumn{4}{c|}{5\% noise level} \\ \hline
    Method                   & $N$ & $u^\ast $ & $u$ & $e_{res}$ & $e_{fit}$ & linTikh eval.\  & $u$ & $e_{res}$ & $e_{fit}$ & linTikh eval.\  \\ \hline \hline
    \multirow{4}{*}{minTikh} & 3 & 2 &  2  &   0.000017     &  0.000055 & 100  & 1.91   &   0.043     &   0.041  & 100   \\ \cline{2-11}
                            & 5 & 2 &  2  &  0.000011      &   0.000026 & 100 & 1.91   &   0.045     &   0.039  & 100   \\ \cline{2-11}
                            & 3 & 8 &  8  &   0.000026     &    0.00028 & 100 &  7.09  &    0.042    &    0.067 & 100   \\ \cline{2-11}
                            & 5 & 8 & 8   &   0.000018     &   0.00022  & 100 & 9.45   &   0.045     &   0.043 & 100   \\ \hline \hline
    \multirow{4}{*}{ADM}    & 3 & 2 &  2.01  &   0.00036     &   0.0030 & 396   &  1.95  &    0.043    & 0.029 & 513 \\ \cline{2-11}
                            & 5 & 2 &  2.01  &    0.00039    &    0.0038 & 445   &  1.95  &   0.045     & 0.025 & 605 \\ \cline{2-11}
                            & 3 & 8 &   7.85  &  0.00015      & 0.007 & 789  & 6.06  &   0.043     &   0.11  & 69   \\ \cline{2-11}
                            & 5 & 8 & 7.86   &  0.00013      &  0.0071  & 948    &  7  &   0.045     & 0.067 & 333 \\ \hline
    \end{tabular}}
    
    \label{table:errortable}
\end{table}

To showcase the applicability to experimental data, the minTikh method was tested on an MRI dataset obtained by a GE Discovery\texttrademark MR750 3.0T MRI scanner, using phase contrast imaging \cite{Wymer_2020} of the internal carotid artery (with point 1 on the common carotid). Here, a 4-dimensional MR data cube is imaged, describing the time-dependent pulse waveforms in each spatial point of the specified field of view. In this case, $255\times255\times5$ spatial points with dimensions $0.7\text{mm} \times 0.7\text{mm} \times 4\text{mm}$, for a total of 102 time steps in one cardiac cycle were imaged. Further, a 3D Time-of-flight (TOF) angiogram is used to track arteries in the 3D image, which is then in turn used to choose data points for the pulse wave analysis presented in this manuscript, as well as for distance calculation between these points. 
Figure~\ref{fig:reala} shows an angiogram of the internal carotids next to the available flow data, where voxels with high flow (i.e., voxels which cover blood vessels) are highlighted in red, the tracked artery is shown as a blue line, and the three data points used for pulse wave analysis are highlighted. In Figure~\ref{fig:realc}, we see the corresponding pulse wave data ($L^1$-normalized to account for \emph{partial voluming}, cf. \cite{Roll_Colchester_Summers_Griffin_1994}), while Figure~\ref{fig:realb} shows the calculated PWV $u_{rec}=7.3\, \text{m/s}$ (vertical green line) and relative residual error in dependence on $u\in U$. Finally, Figure~\ref{fig:reald} shows the calculated split waves, their sum and the measured data in data point $1$. The minTikh method was applied with $N=3$ data points with a total distance of $L_3=14cm$. The parameters in use were manually tuned to $\alpha=0.0001$ and $s=1$. Backed by recent results on the PWV in the carotid artery \cite{Tang_Lee_Chuang_Huang_2020, Darwich_2015}, the results by the minTikh method appear to be realistic for the considered test subject.

\begin{figure}[ht!]
    \centering
    \begin{tabular}{c}
    \subfloat[\centering Visualization of the internal carotid arteries]{{\includegraphics[width=0.47\textwidth,trim={0 0 0 0},clip]{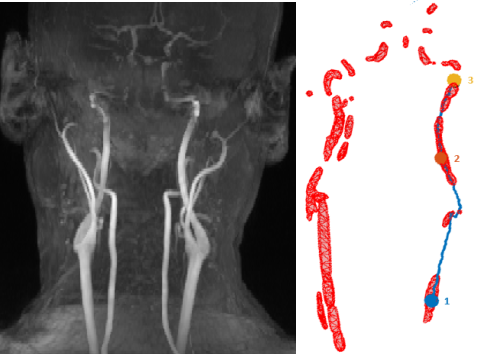} } \label{fig:reala}}%
    \subfloat[\centering reconstructed PWV in m/s and error values]{{\includegraphics[width=0.47\textwidth,trim={0 0 0 0},clip]{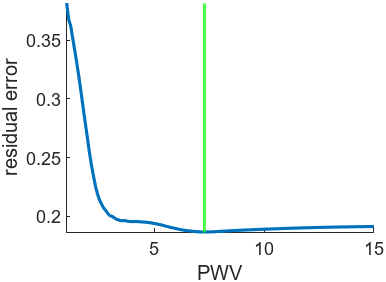} } \label{fig:realb}}%
    \\
    \subfloat[\centering pulse wave data]{{\includegraphics[width=0.47\textwidth,trim={0 0 0 0},clip]{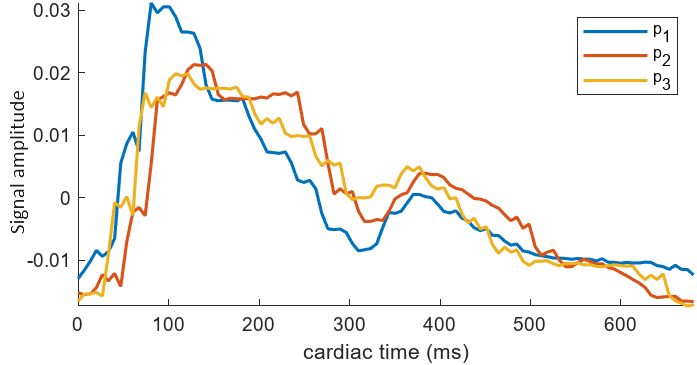} } \label{fig:realc}}%
    \subfloat[\centering resulting split waves in data point $1$]{{\includegraphics[width=0.47\textwidth,trim={0 0 0 0},clip]{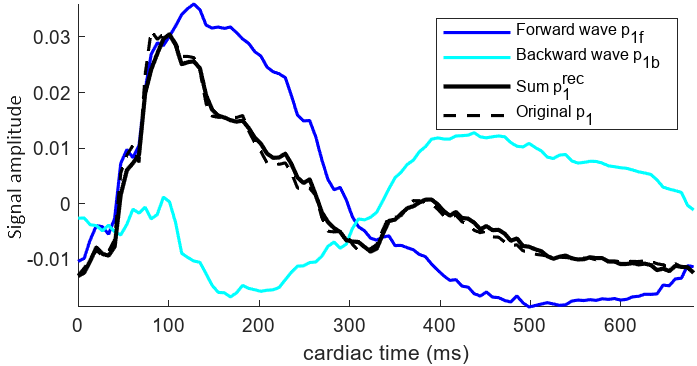} } \label{fig:reald}}%
    \end{tabular}
    \caption{Real data experiment for the minTikh method.}
    \label{fig:real_ex}
\end{figure}

\section{Conclusion and Outlook}\label{sec:conclusion}

In this paper, we developed a mathematical framework for the problem of pulse wave splitting and PWV estimation in the human brain arising in medical applications. We developed two methods to solve the underlying nonlinear and ill-conditioned inverse problem, namely minimizing a linear Tikhonov functional (minTikh) and an alternate direction method (ADM). The algorithms were tested numerically in a simulation environment which aimed to recreate realistic conditions. The results are promising, and show potential for further improvement in less general settings.

In future work, we aim for the proposed methods to be applied to MRI data. However, the results concerning larger pulse wave velocities are an indicator of the ill-conditioned nature of the problem in near real-world conditions. To produce clearer results in this case, either a higher data resolution or a larger spacial difference between data points is required. Note that in the real-world case both these values are restricted due to the resolution of MRI data and the length of the main cerebral arteries in the human brain. Note that the mathematical formulation works completely without a-priori knowledge of the corresponding pulse waves, except a smoothness assumption. Possible real-world knowledge could include that the backward wave is a superposition of the forward wave attenuated at several reflection points. This a-priori information about the pulse waves can be incorporated in the regularization scheme for the application of our methods for real data. Another extension of our model could include the description of pulse waves in larger brain regions, including branches.

In summary, we hope that together with future investigations of required measurement quality and increasing temporal resolution of the data, our proposed method will contribute towards the development of a gold standard for successful PWV estimation.
\appendix
\section{Proof of \cref{Lemma 2.7}}\label{appendixA}

Note that for the subsequent proof, we make use of the Landau notation, i.e.,
we say that $F=\mathcal{O}(g(\|\boldsymbol{h}\|)), $ as $\|\boldsymbol{h}\|\to 0$ for $g:\R\to\R$ if there exists a constant $c>0$ such that $\|F(\boldsymbol{h})\|\leq c g(\|\boldsymbol{h}\|)$ for all $\|\boldsymbol{h}\|$ small enough. 
\begin{proof}
In the first step, we compute the Frech\'et derivative of the operator $A$. This means that we want to find a bounded linear operator $A'(u)$ (cf.~\cite{Deimling_2013}), for which there holds
    \begin{equation}\label{derivative_def}
       A(u+h)-A(u)=A'(u)h+R(u,h)
        \,,
    \end{equation}
with $R(u,{h})=\mathcal{O}\left( |h|^2\right).$
For $A'(u)$ as defined in \eqref{A_derivative}, linearity follows directly from its definition. 
Next, in order to show that $A'(u)$ is bounded and thus continuous, recall that
    \begin{equation*}
        \|A'(u) h\|^2_{\mathcal{Y}_r\to \mathcal{Y}_s} = \sup_{\boldsymbol{y}\in \mathcal{Y}_r}\frac{\|A'(u) h \boldsymbol{y}\|_{\mathcal{Y}_s}^2}{\|\boldsymbol{y}\|^2_{\mathcal{Y}_r}} \,.
    \end{equation*}
To show that this operator norm is finite, consider the estimate
    \begin{equation*}
    \begin{aligned}
        \|A'(u) h \boldsymbol{y}\|^2_{\mathcal{Y}_s}
        &\overset{\phantom{\eqref{whys1}}}{=}
        \sum_{k=1}^N\int_\R(1+|\omega|^2)^s\underbrace{\left|e_k(\omega,u)\right|^2}_{=1} \left|\frac{i\omega L_k h}{u^2}y_k(\omega)\right|^2 \,\mathrm{d}\omega
        \\&\overset{\phantom{\eqref{whys1}}}{\leq}
        \sum_{k=1}^N\left|\frac{L_k h}{u^2}\right |^2\|y_k(\omega)\omega\|^2_{L_{s}^2(\R)}
        \\ &\overset{\eqref{whys1}}{\leq}
        \left|\frac{L_N h}{u^2}\right |^2\|\boldsymbol{y}\|^2_{\mathcal{Y}_{s+1}},
    \end{aligned}    
    \end{equation*}
which yields boundedness of $A'(u)$ via $\|\boldsymbol{y}\|_{\mathcal{Y}_{s+1}}\leq\|\boldsymbol{y}\|_{\mathcal{Y}_r}$ and
    \begin{equation*}
    \begin{aligned}
        \|A'(u) h\|^2_{\mathcal{Y}_r\to \mathcal{Y}_s}&
        \leq
         \sup_{\boldsymbol{y}\in \mathcal{Y}_r}\frac{\left|\frac{L_N h}{u^2}\right |^2\|\boldsymbol{y}\|^2_{\mathcal{Y}_{s+1}}}{\|\boldsymbol{y}\|^2_{\mathcal{Y}_r}}
       \leq\left|\frac{L_N}{u^2} h\right|^2
        \,.
    \end{aligned}
    \end{equation*}
Hence, it remains to show that $A'(u)$ satisfies \eqref{derivative_def} with $R(u,h)=\mathcal{O}(|h|^2).$
For this, we calculate
    \begin{equation*}
    \begin{aligned}
        R(u,h)\boldsymbol{y}(\omega)&=(A(u+h)-A(u)-A'(u) h)\boldsymbol{y}(\omega)
        \\&=
        \,\mathrm{diag}\left(\left(e^{-i\omega \frac{L_k}{u+h}}
        - e^{-i\omega \frac{L_k}{u}}
        -e^{-i\omega\frac{L_k}{u}}\left(\frac{i\omega L_k h}{u^2}\right)  \right)_{k=1}^N\right)\boldsymbol{y}(\omega)
        \\&=\mathrm{diag}\left(\left(e^{-i\omega \frac{L_k}{u+h}}
        - e^{-i\omega \frac{L_k}{u}}
        \left(1+\frac{i\omega L_k h}{u^2}\right)\right)_{k=1}^N\right)\boldsymbol{y}(\omega) \,,  
        \end{aligned}
    \end{equation*}
from which it follows that
    \begin{equation}\label{derivative_norm}
    \begin{aligned}
        \|R(u,h)&\|^2_{\mathcal{Y}_r\to \mathcal{Y}_s}
        \\&=\sup_{\boldsymbol{y}\in \mathcal{Y}_r}\frac{\|R(u,h) \boldsymbol{y}\|^2_{\mathcal{Y}_s}}{\|\boldsymbol{y}\|^2_{\mathcal{Y}_r}}
        \\&=\sup_{\mathcal{Y}_r}\frac{\sum_{k=1}^N\left\|\left(e^{-i\omega \frac{L_k}{u+h}}
        - e^{-i\omega \frac{L_k}{u}}
        \left(1+\frac{i\omega L_k h}{u^2}\right)\right) y_k\right\|^2_{L^2_s(\R)}}{\|\boldsymbol{y}\|^2_{\mathcal{Y}_r}}.
    \end{aligned}
    \end{equation}
Now, define $f(u):=e^{-i\omega L_k/u}$ for all $u\geq\varepsilon>0$ and note that
    \begin{equation}\label{eq_f_derivatives}
        f'(u)=f(u)\frac{i\omega L_k}{u^2} \qquad \text{and} \qquad f''(u)=f(u)\left(\frac{L_k^2\omega^2+2iu L_k\omega}{u^4}\right)\,.
    \end{equation}
Using this, we can now estimate
    \begin{equation}\label{derivative_limit2}
    \begin{aligned}
        &\left\|y_k\left(e^{-i\omega \frac{L_k}{u+h}}-e^{-i\omega \frac{L_k}{u}}\left(1+\frac{i\omega L_k h}{u^2}\right) \right) \right\|_{L_s^2(\R)}
        \\ &\qquad=
       |h|\left\|y_k\left(\frac{e^{-i\omega \frac{L_k}{u+h}}-e^{-i\omega \frac{L_k}{u}}}{h}-e^{-i\omega \frac{L_k}{u}}\frac{i\omega L_k}{u^2}\right)\right\|_{L_s^2(\R)}
        \\ &\qquad=
       |h|\left\|y_k\left(\frac{f(u+h)-f(u)}{h}-f'(u)\right)\right\|_{L_s^2(\R)} \,.
    \end{aligned}
    \end{equation}
Due to \eqref{eq_f_derivatives} $f$ is twice continuously differentiable for $u \geq \eps > 0$. Thus, Taylor's theorem yields $f(u+h)=f(u)+hf'(u)+\frac{h^2}{2}f''(\xi(h))$ for some $\xi = \xi(h)\in [u,u+h]$. Together with $\abs{f(\xi(h))}=1$ it thus follows that
    \begin{equation}\label{der_proof3}
    \begin{aligned}
        |h|&\left\|y_k\left(\frac{f(u+h)-f(u)}{h}-f'(u)\right)\right\|_{L_s^2(\R)}
        \\&\overset{\phantom{\eqref{whys1}}}{=}
        |h|^2\left\|\frac{y_k}{2}f''(\xi(h))\right\|_{L_s^2(\R)}
        \\&\overset{\phantom{\eqref{whys1}}}{=}
        |h|^2\left\|\frac{y_k}{2}f(\xi(h))\left(\frac{L_k^2\omega^2+2i\xi(h) L_k\omega}{\xi(h)^4}\right)\right\|_{L_s^2(\R)}
        \\&\overset{\phantom{\eqref{whys1}}}{\leq}
        |h|^2\left(\frac{L_k^2}{2\xi(h)^4}\left\|y_k(\omega)\omega^2\right\|_{L_s^2(\R)}
        +\frac{L_k}{\xi(h)^3}\left\|y_k(\omega)\omega\right\|_{L_s^2(\R)}\right)
        \\&\overset{\eqref{whys1}}{\leq} |h|^2\left(\frac{L_k^2}{2u^4}
        +\frac{L_k}{u^3}\right)\left\|y_k\right\|_{L_{s+2}^2(\R)} \,,
    \end{aligned}
    \end{equation}
which holds for all $y\in L^2_{r}(\R)$.
Combining \eqref{derivative_norm}, \eqref{derivative_limit2} and \eqref{der_proof3} we obtain
\begin{equation*}
\begin{aligned}
        \|R(u,h)\|_{\mathcal{Y}_r\to \mathcal{Y}_s}^2
    &\leq\sup_{y\in \mathcal{Y}_r}\frac{\sum_{k=1}^N |h|^4\left(\frac{L_k^2}{2u^4}+\frac{L_k}{u^3}\right)^2\left\|y_k\right\|^2_{L_{s+2}^2(\R)} }{\|\boldsymbol{y}\|^2_{\mathcal{Y}_r}}
        \\&\leq
        |h|^4\left(\frac{L_N^2}{2u^4}+\frac{L_N}{u^3}\right)^2,
\end{aligned}
\end{equation*}
which implies $R(u,h)=\mathcal{O}(|h|^2)$. This now yields that $A'(u)h$ as in \eqref{A_derivative} is indeed the Fr\'echet derivative of $A$. The Frech\'et derivative of $B$ follows analogously, which concludes the proof.
\end{proof}

\section*{Acknowledgement}

This research was funded in part by the Austrian Science Fund (FWF) SFB 10.55776/F68 ``Tomography Across the Scales'', project F6805-N36 (Tomography in Astronomy). For open access purposes, the authors have applied a CC BY public copyright license to any author-accepted manuscript version arising from this submission. LW is partially supported by the State of Upper Austria. HUV is partially funded by NIH grant \# 5 R21 EY027568 02. This research cites an algorithm described in a patent application by the Center for Technology Licensing at Cornell University (CTL), of which one of the authors (HUV) is the inventor.

\bibliographystyle{plain}
{\footnotesize
\bibliography{references(1)}
}

\end{document}